\documentclass[10pt, reqno, oneside]{article}

\usepackage{amsmath,amsthm,amsfonts,amssymb}

\title{\textbf{A cutting surface algorithm for semi-infinite convex programming with an application to moment robust optimization}}
\date{\today}
\author{Sanjay Mehrotra\footnote{Northwestern University, Department of Industrial Engineering and Management Sciences, Evanston, IL, USA. E-mail: mehrotra@iems.northwestern.edu} \and D{\'a}vid Papp\footnote{Northwestern University, Department of Industrial Engineering and Management Sciences, Evanston, IL, USA. Email: dpapp@iems.northwestern.edu. Currently at Harvard Medical School and Massachusetts General Hospital.}}
\fussy

\usepackage[round]{natbib}
\usepackage[margin=1.3in]{geometry}
\usepackage{graphicx}
\usepackage{booktabs}
\usepackage{multirow}
\usepackage[colorlinks=true, citecolor=black, linkcolor=black, urlcolor=black, setpagesize=false, plainpages=false, pdfpagelabels]{hyperref}
\usepackage{enumerate}
\usepackage{subfigure}
\usepackage{mdframed}

\newcommand{\real}{\mathbb{R}}

\newcommand{\T}{\mathrm{T}} % for transpose
\newcommand{\defeq}{\ensuremath{\overset{\mathrm{def}}{=}}}

\DeclareMathOperator*{\argmax}{arg\,max}
\DeclareMathOperator*{\dom}{dom}
\renewcommand{\epsilon}{\varepsilon}

\newcommand{\deletethis}[1]{}

\theoremstyle{plain}
\newtheorem{definition}{Definition}
\newtheorem{lemma}[definition]{Lemma}

\newtheorem{theorem}[definition]{Theorem}

\theoremstyle{definition}
\newtheorem{example}{Example}
\newtheorem{algorithm}{Algorithm}
\newtheorem{assumption}{Assumption}

\begin{document}
\maketitle

\begin{abstract}
We present and analyze a central cutting surface algorithm for general semi-infinite convex optimization problems, and use it to develop a novel algorithm for distributionally robust optimization problems in which the uncertainty set consists of probability distributions with given bounds on their moments. Moments of arbitrary order, as well as non-polynomial moments can be included in the formulation. We show that this gives rise to a hierarchy of optimization problems with decreasing levels of risk-aversion, with classic robust optimization at one end of the spectrum, and stochastic programming at the other. Although our primary motivation is to solve distributionally robust optimization problems with moment uncertainty, the cutting surface method for general semi-infinite convex programs is also of independent interest. The proposed method is applicable to problems with non-differentiable semi-infinite constraints indexed by an infinite-dimensional index set. Examples comparing the cutting surface algorithm to the central cutting plane algorithm of Kortanek and No demonstrate the potential of our algorithm even in the solution of traditional semi-infinite convex programming problems, whose constraints are differentiable, and are indexed by an index set of low dimension. After the rate of convergence analysis of the cutting surface algorithm, we extend the authors' moment matching scenario generation algorithm to a probabilistic algorithm that finds optimal probability distributions subject to moment constraints. The combination of this distribution optimization method and the central cutting surface algorithm yields a solution to a family of distributionally robust optimization problems that are considerably more general than the ones proposed to date.

\end{abstract}

\section{Introduction}
We present a novel cutting surface algorithm for general semi-infinite convex optimization problems (SICPs) that is applicable under milder than usual assumptions on the problem formulation, extending an algorithm of \citet{KortanekNo1993}. Our primary motivation is to solve a large class of distributionally robust optimization problems that can be posed as SICPs with convex but not necessarily differentiable constraints indexed by an uncountably infinite dimensional set of probability distributions. In the remainder of this section we introduce the SICPs considered; the connection to robust optimization is discussed in Section \ref{sec:DRO}.

We consider a general semi-infinite convex optimization problem of the following form:
\begin{equation}
\begin{split}
\textrm{minimize}   &\quad x_0\\
\textrm{subject to} &\quad g(x,t)\leq 0\quad \forall\,t\in T\\
                    &\quad x \in X
\end{split}\tag{SICP}\label{eq:SICP}
\end{equation}
with respect to the decision variables $x$ (whose first coordinate is denoted by $x_0$), where the sets $X$ and $T$, and the function $g\colon X\times T \mapsto \real$ satisfy the following conditions:
\begin{assumption}\label{ass:1}\mbox{}
\begin{enumerate}
\item the set $X\subseteq\real^n$ is convex, closed and bounded;
\item there exists a \emph{Slater point} $\bar{x}$ and $\eta>0$ satisfying $\bar{x}\in X$ and $g(\bar{x},t) \leq -\eta$ for every $t\in T$;
\item the function $g(\cdot,t)$ is convex and subdifferentiable for every $t\in T$; moreover, these subdifferentials are uniformly bounded: there exists a $B>0$ such that for every $x\in X$ and $t \in T$, every subgradient $d\in\partial g_x(x,t)$ satisfies $\|d\|\leq B$.
\end{enumerate}
\end{assumption}

Note that having one of the components of the variable vector $x$ as an objective instead of a general convex objective function is without loss of generality; we opted for this form because it simplifies both the description of our algorithm and the convergence analysis. Similarly, we can assume without loss of generality that $\eta=1$ in the assumption; otherwise we can simply replace $g$ by $g/\eta$. (This will, of course, change the value of $B$ as well.) We also remark that $T$ is not required to be either convex or finite dimensional, nor is the differentiability of $g$, or the convexity or concavity of $g$ in its second argument necessary.

The minimum of \eqref{eq:SICP} is attained, since its feasible set is closed, nonempty, and bounded, and its objective function is continuous. Our aim is to find an optimal solution to \eqref{eq:SICP} within $\epsilon$ accuracy, by which we mean the following.
%We say that a point $x \in X$ is \emph{$\epsilon$-feasible} if $g(x,t)\leq -\epsilon$ for every $t\in T$, similarly $x \in X$ is \emph{$\epsilon$-infeasible} if $g(x,t)\geq\epsilon$ for some $t \in T$,
%and we say that a point $x^*\in X$ is an \emph{$\epsilon$-optimal} solution to \eqref{eq:SICP} if
%\[x^*_0 \leq \min \{x_0\,|\,g(x,t)\leq-\epsilon\; \forall\,t\in T \}.\]
%The right-hand side of this inequality is well-defined for every $\epsilon\in[0,1]$.
%The right-hand side of this inequality is well-defined for every $\epsilon\in[0,1]$.

We say that $x \in X$ is \emph{$\epsilon$-feasible} if $g(x,t)\leq\epsilon$ for every $t \in T$, and we say that a point $x_\epsilon^*\in X$ is an \emph{$\epsilon$-optimal solution} to \eqref{eq:SICP} if it is $\epsilon$-feasible and
\[(x_\epsilon^*)_0 \leq x^*_0 \defeq \min \{x_0\,|\,x \in X,\; g(x,t)\leq 0\, \forall\,t\in T \}.\]

We make one final assumption, on our ability to detect the approximate infeasibility of candidate solutions to \eqref{eq:SICP} within a prescribed error $\epsilon\geq 0$.
\begin{assumption}\label{ass:2}For every point $x\in X$ that is not $\epsilon$-feasible, we can find in finite time a $t\in T$ satisfying $g(x,t)>0$.
\end{assumption}

It is \emph{not} required that we can find the most violated inequality $g(x,t)>0$ or the corresponding $\argmax_{t\in T}\{g(x,t)\}$ for any $x$. We only stipulate that we shall be able to find \emph{a} violated inequality, provided that some inequality is violated by more than $\epsilon$. %The algorithm provided in Section \ref{sec:algorithm} computes an $\epsilon$-optimal solution to \eqref{eq:SICP} as long as Assumption \ref{ass:2} is satisfied with the same $\epsilon$.

Assumption \ref{ass:2} is slightly weaker than the more ``natural'' assumption of having an oracle that either returns a $t\in T$ satisfying $g(x,t)>\epsilon$ or concludes that $g(x,t)\leq \epsilon$ for all $t\in T$ for some fixed $\epsilon\geq 0$. Our form is motivated by the moment robust optimization application. As we shall see in Section \ref{sec:CG}, the cut generation oracle for moment robust optimization needs to solve an (infinite dimensional) distribution optimization problem that is difficult to solve exactly; however, the method presented in that section guarantees that if a sufficiently violated constraint exists, then \emph{some} violated constraint is found in \textit{a priori} bounded time.

Several algorithms have been proposed to solve semi-infinite linear and semi-infinite convex programming problems, including cutting plane methods, %such as \citep{KortanekNo1993} and \citep{Betro2004},
local reduction methods, exchange methods, and homotopy methods. See, for example, \citep{LopezStill2007} for a recent review on semi-infinite convex programming, including an overview on numerical methods with plenty of references. Most existing algorithms consider only linear problems, appealing to the fact that the general convex problem \eqref{eq:SICP} is equivalent to the semi-infinite linear programming problem
\begin{equation}
\begin{split}
\textrm{minimize}   &\quad x_0\\
\textrm{subject to} &\quad u^\T x - g_t^*(u)\leq 0\quad \forall\,t\in T\text{ and } u\in \dom g_t^*\\
                    &\quad x \in X,
\end{split}\tag{SILP}\label{eq:SILP}
\end{equation}
where $g_t^*$ denotes the conjugate function of $g(\cdot,t)$. We contend, however, that this transformation is usually very ineffective, because if $X$ is $n$-dimensional, $T$ is $d$-dimensional, and (as it is very often the case) $d\ll n$, then the index set in the semi-infinite constraint set increases from $d$ to the considerably higher $d+n$. Also, the set $T$ and the function $g$ might have special properties that allow us to find violated inequalities $g(x,t)\leq 0$ relatively easily; a property that may not be inherited by the set $\{(t,u)\,|\,t\in T, u\in \dom g_t^*\}$ and the conjugate function $g^*$ in the inequality constraints of \eqref{eq:SILP}. This is also the case in our motivating application. For such problems, the use of non-linear convex cuts (sometimes called \emph{cutting surfaces}) generated directly from the original convex problem \eqref{eq:SICP} is preferred to the use of cutting planes generated from the equivalent linear formulation \eqref{eq:SILP}. 

Another family of semi-infinite convex problems where the use of cutting surfaces is more attractive than the use of cutting planes consists of problems where $X$
is a high-dimensional non-polyhedral set, whose polyhedral approximation to $X$ is expensive to construct. In this case, any
advantage gained from the linear reformulation of the semi-infinite constraints disappears, as \eqref{eq:SILP} still remains a
nonlinear convex program. Even if $X$ is polyhedral, and only the constraints $g$ are non-linear, cutting surfaces can be
attractive in the high-dimensional case, where a cutting plane method may require a large number of cuts to obtain a sufficiently
good polyhedral approximation of the non-linear constraints in the vicinity of the optimum. The trade-off between having to solve
a large number of linear master problems \textit{versus} having to solve a small number of non-linear convex master problems is
not clear, but rather problem-dependent. Example \ref{ex:4e} (in Section \ref{sec:numerical}) presents a case where the cutting
surface method scales well with the increasing dimensionality of the optimization problem, whereas the cutting plane method
breaks down.

Our algorithm is motivated by the ``central cutting plane'' algorithm of \citep{KortanekNo1993} for convex problems, which in turn is an extension of Gribik's algorithm \citep{Gribik1979}. Gribik's algorithm has been the prototype of several cutting plane algorithms in the field, and has been improved in various ways, such as in the ``accelerated central cutting plane'' method of \citep{Betro2004}. Our algorithm can also be viewed as a modification of a traditional convex constraint generation method, in which the restricted master problem attempts to drive its optimal solutions towards the center of the current outer approximation of the feasible set. The traditional constraint generation method is a special case of our algorithm with all centering parameters set to zero.

Our main contribution from the perspective of semi-infinite programming is that we extend the central cutting plane algorithm to a cutting surface algorithm allowing non-linear convex cuts. The possibility of dropping cuts is retained, although in our numerical examples we always found optimal solutions very quickly, before dropping cuts was necessary for efficiency. %Although cutting surface algorithms have a general convex master problem to solve in each iteration instead of a linear programming problem, this difference diminishes in the presence of other convex constraints defining $X$ or a non-linear objective function.

The outline of the paper is as follows. Distributionally robust optimization is reviewed in \mbox{Section \ref{sec:DRO}}, where
we also give a semi-infinite convex formulation of this problem, and state our result on the convergence of the optimum objective
value of the moment robust problems to that of stochastic programs.  We proceed by describing our cutting surface algorithm for
semi-infinite convex programming in \mbox{Section \ref{sec:algorithm}}, and proving its correctness and analyzing its rate of
convergence in \mbox{Section \ref{sec:convergence}}. The application of this method to distributionally robust optimization
requires a specialized column generation method, which is introduced in Section \ref{sec:CG}. Computational results, which
include both standard semi-infinite convex benchmark problems and distributionally robust utility maximization problems, follow
in \mbox{Section \ref{sec:numerical}}; with concluding remarks in \mbox{Section \ref{sec:conclusion}}.

\section{Distributionally robust and moment robust optimization}\label{sec:DRO}
Stochastic optimization and robust optimization are two families of optimization models introduced to tackle decision making
problems with uncertain data. Broadly speaking, robust optimization handles the uncertainty by optimizing for the worst case
within a prescribed set of scenarios, whereas stochastic optimization assumes that the uncertain data follows a specified
probability distribution. \emph{Distributionally robust optimization}, introduced in \citep{Scarf1957}, can be seen as a
combination of these approaches, where the optimal decisions are sought for the worst case within a prescribed set of probability
distributions that the data might follow. The term \emph{robust stochastic programming} is also often used to describe
optimization models of the same form.

Formally, let the uncertain data be described by a random variable supported on a set $\Xi\subseteq\real^d$, following an \emph{unknown} distribution $P$ from a set of probability distributions $\mathfrak{P}$. Then a general distributionally robust optimization problem is an optimization model of the form
\begin{equation}
\min_{x\in X} \max_{P\in\mathfrak{P}} \mathbb{E}_P[H(x)],\; \text{ or (equivalently) }\; \min_{x\in X} \max_{P\in\mathfrak{P}} \int_{\xi\in\Xi} h(x,\xi) P(d\xi),\tag{DRO}\label{eq:DRO}
\end{equation}
where $H$ is a random cost or disutility function we seek to minimize in expectation, $h$ is the corresponding weight function in the equivalent integral form; the argument $x$ of $H$ and $h$ is our decision vector. We assume that all expectations (integrals) exist and that the minima and maxima are well-defined. We shall also assume for the rest of the paper that the support set $\Xi$ is closed and bounded.

With the above notation, a general stochastic optimization problem is simply \eqref{eq:DRO} with a singleton $\mathfrak{P}$,
while a standard robust optimization problem is \eqref{eq:DRO} with a set $\mathfrak{P}$ that consists of all probability
distributions supported on a point in $\Xi$.

One can also view the general distributionally robust optimization problem not only as a common generalization of robust and stochastic optimization, but also as an optimization model with an adjustable level of risk-aversion. To see this, consider a nested sequence of sets of probability distributions $\mathfrak{P}_0 \supseteq \mathfrak{P}_1 \supseteq \cdots$, where $\mathfrak{P}_0$ is the set of all probability distributions supported on $\Xi$, and $\mathfrak{P}_\infty\defeq\cap_{i=0}^\infty \mathfrak{P}_i$ is a singleton set. In the corresponding sequence of problems \eqref{eq:DRO}, the first one is the classic robust optimization problem, which is the most conservative (risk-averse) of all, optimizing against the worst case, and the last one is the classic stochastic optimization problem, where the optimization is against a fixed distribution. At the intermediate levels the models correspond to decreasing levels of risk-aversion.

Such a sequence of problems can be constructed in many natural ways; we shall only focus on the case when the sequence of $\mathfrak{P}_i$'s is defined by constraining an increasing number of moments of the
underlying probability distribution. In this case the \eqref{eq:DRO} problem is called \emph{moment robust optimization} problem.
Theorem \ref{thm:DRO-to-SP} below establishes the ``convergence'' of this sequence of moment-robust optimization problems to a stochastic optimization problem for closed and bounded domains $\Xi$.

\begin{theorem}\label{thm:DRO-to-SP}
Let $h$ and $X$ as above, assume that $\Xi$ is closed and bounded, and that $h$ is continuous. Let $P$ be a probability distribution supported on $\Xi$, with moments $m_k(P) \defeq \int_\Xi \xi_1^{k_1}\cdots \xi_n^{k_n} P(d\xi)$. For each $i=0,1,\dots$, let $\mathfrak{P}_i$ denote the set of probability distributions $Q$ supported on $\Xi$ whose moments $m_k(Q)$ satisfy $m_k(Q) = m_k(P)$ for every multi-index $k$ with $0 \leq k_1+\dots+k_n \leq i$. Finally, for each $i=0,1,\dots$ define the moment-robust optimization problem \eqref{eq:DROi} as follows:
\begin{equation}\label{eq:DROi}
\min_{x\in X} \max_{Q\in\mathfrak{P}_i} \int_{\xi\in\Xi} h(x,\xi) Q(d\xi). \tag{DRO$_i$}
\end{equation}
Then the sequence of the optimal objective function values of \eqref{eq:DROi} converges to the optimal objective function value of the stochastic program
\begin{equation}\label{eq:SP}
\min_{x\in X} \int_{\xi\in\Xi}h(x,\xi)P(d\xi). \tag{SP}
\end{equation}
\end{theorem}

The proof is given in the Appendix.

It is interesting to note that in the above theorem the function $h(x,\cdot)$ could be replaced by any continuous function
$f\colon\Xi\mapsto\real$ that does not depend on $x$, proving that
\[ \lim_{i\to\infty} \int_{\xi\in\Xi}f(\xi)\bar{Q}_i(d\xi) = \int_{\xi\in\Xi}f(\xi)P(d\xi) \]
for \emph{every continuous function} $f\colon\Xi\mapsto\real$; in other words, the sequence of measures $\bar{Q}_0, \bar{Q}_1,
\dots$ converges weakly to $P$, and so does every other sequence of measures in which the moments of the $i$th measure agree with
the moments of $P$ up to order $i$. Therefore, Theorem \ref{thm:DRO-to-SP} can be seen as a generalization of the well-known
theorem that the moments of a probability distribution with compact support uniquely determine the distribution. For
distributions with unbounded support, a statement similar to Theorem \ref{thm:DRO-to-SP} can only be made if the moments in
question uniquely determine the probability distribution $P$. A collection of sufficient conditions under which infinite moment
sequences determine a distribution can be found in the recent review article \citep{KleiberStoyanov2013}.

In a more realistic, data-driven setting, bounds on the moments of uncertain data can be obtained by computing confidence intervals around the sample moments of the empirical distribution, and by application-specific considerations, such as a measurement or other error having mean zero.

\subsection{Past work}
In most applications since Scarf's pioneering work \citep{Scarf1957}, the set of distributions $\mathfrak{P}$ is defined by setting bounds on the moments of $P$; recent examples include \citep{DelageYe2010}, \citep{Bertsimas2010models}, and \citep{MehrotraZhang2013}. Simple lower and upper bounds (confidence intervals and ellipsoids) on moments of arbitrary order are easily obtained using standard statistical methods; \citep{DelageYe2010} describes an alternative method to derive bounds on the first and second moments. However, to the best of our knowledge, no algorithm has been proposed until now to solve \eqref{eq:DRO} with sets $\mathfrak{P}$ defined by constraints on moments of order higher than two.

Recent research has focused on conditions under which \eqref{eq:DRO} with moment constraints can be solved in polynomial time. \citet{DelageYe2010} consider an uncertainty set defined via a novel type of confidence set around the mean vector and covariance matrix, and show that \eqref{eq:DRO} with uncertainty sets of this type can be solved in polynomial time (using the ellipsoid method) for a class of probility mass functions $h$ that are convex in $x$ but concave in $\xi$. \citet{MehrotraZhang2013} extend this result by providing polynomial time methods (using semidefinite programming) for least squares problems, which are convex in both $x$ and $\xi$. The uncertainty sets in their formulation are defined through bounds on the measure, bounds on the distance from a reference measure, and moment constraints of the same form as considered in \citep{DelageYe2010}. \citet{Bertsimas2010models} consider two-stage robust stochastic models in which risk aversion is modeled in a moment robust framework using first and second order moments.

Our method is not polynomial time, but it can be applied to problems where bounds of moments of arbitrary order (and possibly bounds on non-polynomial moments) are available. This allows the decision maker to shape the distributions in $\mathfrak{P}$ better. Moments up to order $4$ are easily interpretable and have been used to strengthen the formulation of stochastic programming models. \citep{HKW-03} provides a heuristic to improve stochastic programming models using first and second order moments as well as marginal moments up to order $4$.

Our approach is based on a semi-infinite convex reformulation of \eqref{eq:DRO}, which is discussed next.

\subsection{Distributionally robust optimization as a semi-infinite convex program}

Consider the second (integral) form of \eqref{eq:DRO} with a function $h$ that is convex in $x$ for every $\xi$. If $\Xi$ and $X$ are bounded sets, the optimal objective function value can be bracketed in an interval $[z_\text{min},z_\text{max}]$, and the problem can be written as a semi-infinite convex optimization problem
\begin{equation}
\begin{split}
\textrm{minimize}   &\quad z\\
\textrm{subject to} &\quad -z + \int_\Xi h(x,\xi) P(d\xi) \leq 0 \quad \forall\,P\in \mathfrak{P}\\
                    &\quad (z,x) \in [z_\text{min},z_\text{max}] \times X,
\end{split}\label{eq:DRO-SICP}
\end{equation}
which is a problem of the form \eqref{eq:SICP}; the set $\mathfrak{P}$ plays the role of $T$; $z$ plays the role of $x_0$. Note
that in the above problem the index set of the constraints is not a low-dimensional set, as it is common in semi-infinite convex
programming, but an infinite dimensional set. Therefore, we cannot assume without further justification that violated
inequalities in \eqref{eq:SICP} can be easily found.

It can be verified, however, that this problem satisfies Assumption \ref{ass:1} as long as $h$ has bounded subdifferentials on the boundary of $\Xi$. \mbox{Assumption \ref{ass:2}} for \eqref{eq:DRO-SICP} means that for the current best estimate $z^{(k)}$ of the optimal $z$ given by the algorithm, we can find a $P$ such that $\int_\Xi h(x,\xi) P(d\xi) > z^{(k)}$ provided that there is a $P$ for which $\int_\Xi h(x,\xi) P(d\xi) > z^{(k)} + \epsilon$. As $z^{(k)}$ approaches the optimal value of the integral, \mbox{Assumption \ref{ass:2}} gradually translates to being able to find
%the following question: for a prescribed $\epsilon>0$, given $z$ and $x$ such that $z + \epsilon < \int_{\xi\in\Xi} h(x,\xi) P(d\xi)$ for some $P\in\mathfrak{P}$, find a $P\in\mathfrak{P}$ satisfying $z < \int_{\xi\in\Xi} h(x,\xi) P(d\xi)$. This is the same problem as solving
\begin{equation}\label{eq:maxP}
\sup_{P\in\mathfrak{P}} \int_\Xi h(x,\xi) P(d\xi)
\end{equation}
(in which $x$ is a parameter) within a prescribed $\epsilon>0$ error. We shall concentrate on this problem in the context of
moment robust optimization in Section~\ref{sec:CG}.

In the moment-robust formulation of \eqref{eq:DRO} the set $\mathfrak{P}$ is defined via bounds on some (not necessarily polynomial) moments: given continuous $\Xi\mapsto\real$ \emph{basis functions} $f_1,\dots,f_N$, and a lower and upper bound vector $\ell$ and $u$ on the corresponding moments, we set
\begin{equation}\label{eq:defP}
\mathfrak{P} = \left\{P\;\middle|\;\int_\Xi f_i(\xi) P(d\xi) \in [\ell_i, u_i],\, i=1,\dots,N \right\}.
\end{equation}
In typical applications the $f_i$ form a basis of low-degree polynomials. For example, if we wish to optimize for the worst-case
distribution among distributions having prescribed mean vector and covariance matrix, then  $f_i$ can be the $n$-variate
monomials up to degree two (including the constant $1$ function), and $\ell=u$ is the vector of prescribed moments (including the
``zeroth moment'', $1$).

% the "concussion" algorithm :-)
\section{A central cutting surface algorithm for semi-infinite convex programming}\label{sec:algorithm}

The pseudo-code of our cutting surface algorithm is given in Algorithm \ref{alg:CONCUSSION}. A few remarks are in order before we proceed to proving its correctness.

First, we assume that the instance of \eqref{eq:SICP} that we wish to solve satisfies \mbox{Assumptions \ref{ass:1} and \ref{ass:2}}. The algorithm also applies to the semi-infinite formulation \eqref{eq:DRO-SICP} of distributionally robust optimization. In that context, \mbox{Assumption \ref{ass:1}} is satisfied as long as $h$ has bounded subdifferentials on the boundary of $\Xi$. As discussed in the previous section, \mbox{Assumption \ref{ass:2}} translates to being able to find $\epsilon$-optimal solutions to problems of the form \eqref{eq:maxP}.

Second, by correctness of \mbox{Algorithm \ref{alg:CONCUSSION}} it is meant that the algorithm computes an $\epsilon$-optimal solution to \eqref{eq:SICP} as long as Assumption \ref{ass:2} is satisfied with the same $\epsilon$.

Throughout the algorithm, $y^{(k-1)}$ is the best $\epsilon$-feasible solution found so far (or the initial vector $y^{(0)}$), and its first coordinate, $y^{(k-1)}_0$ is an upper bound on the objective function value of the best $\epsilon$-feasible point.  The initial value of $y^{(0)}_0$ is an arbitrary upper bound $U$ on this optimum; the other components of $y^{(0)}$ may be initialized arbitrarily.

%To define our algorithm (which is really a family of algorithms) we need the following definition: we say that a function $h\colon\real^n\mapsto\real$ is a \emph{slack function} if it has the following two properties: (1) $0 < h(v) < H_1$ for some fixed $H_1\in\real$ for every non-zero vector $v$, and (2) $\|v\|/h(v) < H_2$ for some fixed $H_2\in\real$ for every non-zero vector $v$.

%The simplest slack functions are positive constants, we shall use another slack function in Algorithm \ref{alg:CONCUSSION} only if we can compute a subgradient of $g(\cdot,t)$ at every point, for every $t\in T$; in this case another example of a slack function is $h(v) = \max(\|v\|,1)$. Different slack functions can be used in the algorithm in order to scale the slacks of satisfied constraints using gradient information, whenever such information is available. This may impact the practical performance of the algorithm.

In Step 2 of the algorithm we attempt to improve on the current upper bound by as much as possible and identify a ``central'' point $x^{(k)}$ that satisfies all the added inequalities with a large slack. The algorithm stops in Step 3 when no such improvement is possible.

In each iteration $k$, either a new cut is added in Step 5 that cuts off the last, infeasible, $x^{(k)}$ (a \emph{feasibility cut}), or it is found that $x^{(k)}$ is an $\epsilon$-feasible solution, and the best found $\epsilon$-feasible solution $y^{(k)}$ is updated in Step 6 (an \emph{optimality cut}). In either case, some inactive cuts are dropped in the optional Step 7. The parameter $\beta$ adjusts how aggressively cuts are dropped; setting $\beta = \infty$ is equivalent to skipping this step altogether.

In Step 5 of every iteration $k$ a centering parameter $s^{(k)}$ needs to be chosen. To ensure convergence of the method, it is sufficient that this parameter is bounded away from zero, and that it is bounded from above: $s_{\min} \leq s^{(k)} \leq B$ for every $k$, with some $s_{\min} > 0$. (It is without loss of generality that we use the same upper bound as we used for the subgradient norms.) Another strategy that ensures convergence is to find a subgradient $d\in\partial_x g(x^{(k)},t^{(k)})$ and set $s^{(k)} = \alpha \|d\|$ with an arbitrary $\alpha\in (0,1]$, which will give positive values for the centering parameter, but is not necessarily bounded away from zero. Below we prove that Algorithm \ref{alg:CONCUSSION} converges in all of these cases.

\begin{figure}[!ht]
\begin{mdframed}[innerleftmargin=2em]
\begin{algorithm}[Central cutting surface algorithm]\label{alg:CONCUSSION}
\hfill\par
\begin{enumerate}
\item[] Parameters: a strict upper bound $U$ on the optimal objective function value of \eqref{eq:SICP}; a $B>0$ for which Assumption \ref{ass:1} holds; a tolerance $\epsilon\geq 0$ for which Assumption \ref{ass:2} holds; and an arbitrary $\beta > 1$ specifying how aggressively cuts are dropped.
\item[Step 1.] \textbf{(Initialization.)} Set $k=1$, $y^{(0)}=(U,0,\dots,0)\in\real^n$, and $J^{(0)} = \emptyset$.
%Find an $x^{(0)} \in \argmin_x \{ x_0\colon x\in X\}$.
\item[Step 2.] \textbf{(Solve master problem.)} Determine the optimal solution $(x^{(k)},\sigma^{(k)})$ of the optimization problem
\begin{equation}\label{eq:master}
\begin{split}
\textrm{maximize}   &\quad \sigma\\
\textrm{subject to} &\quad x_0 + \sigma \leq y_0^{(k-1)}\\
                    &\quad g(x,t^{(j)}) + \sigma s^{(j)} \leq 0\quad \forall\,j\in J^{(k-1)}\\
                    &\quad x \in X.
\end{split}
\end{equation}
\item[Step 3.] \textbf{(Optimal solution?)} If $\sigma^{(k)} = 0$, stop and return $y^{(k-1)}$.
\item[Step 4.] \textbf{(Feasible solution?)} Find a $t^{(k)}\in T$ satisfying $g(x^{(k)},t^{(k)})>0$ if possible.\\
If no such $t^{(k)}$ is found, go to Step 6.
\item[Step 5.] \textbf{(Feasibility cut.)} Set $J^{(k)} = J^{(k-1)}\cup \{k\}$ and $y^{(k)} = y^{(k-1)}$; choose a centering parameter $s_{\min}\leq s^{(k)}\leq B$. (See the text for different strategies.)\\
Go to Step 7.
\item[Step 6.] \textbf{(Optimality cut; update best known $\epsilon$-feasible solution.)} Set  $J^{(k)} = J^{(k-1)}$ and $y^{(k)} = x^{(k)}$.
\item[Step 7.] \textbf{(Drop cuts.)} Let $D = \{j\,|\,\sigma^{(j)} \geq \beta\sigma^{(k)} \text{ and } g(x^{(k)}) + \sigma^{(k)}s^{(j)} < 0\}$, and set $J^{(k)} = J^{(k)}\setminus D$.
\item[Step 8.] Increase $k$ by one, and go to Step 2.
\end{enumerate}
\end{algorithm}
\end{mdframed}
\end{figure}

\section{Correctness and rate of convergence}\label{sec:convergence}

We show the correctness of the algorithm by proving the following theorems. We tacitly assume that the centering parameters $s^{(k)}$ are chosen in Step 5 according to one of the two strategies mentioned above.

\begin{theorem}\label{thm:1}
Suppose that \mbox{Algorithm \ref{alg:CONCUSSION}} terminates in the $k$th iteration. Then $y^{(k-1)}$ is an $\epsilon$-optimal solution to \eqref{eq:SICP}.
\end{theorem}

\begin{theorem}\label{thm:2}
Suppose that \mbox{Algorithm \ref{alg:CONCUSSION}} does \emph{not} terminate. Then there exists an index $\hat{k}$ such that the sequence $(y^{(\hat{k}+i)})_{i=1,2,\dots}$ consists entirely of $\epsilon$-feasible solutions.
\end{theorem}

\begin{theorem}\label{thm:3}
Suppose that \mbox{Algorithm \ref{alg:CONCUSSION}} does \emph{not} terminate. Then the sequence $(y^{(k)})_{k=1,2,\dots}$ has an accumulation point, and each accumulation point is an $\epsilon$-optimal solution to \eqref{eq:SICP}.
\end{theorem}

\noindent Therefore, the algorithm either finds an $\epsilon$-optimal solution after finitely many iterations, or approaches one in the limit. Even in the second case, the $\epsilon$-optimal solution is approached through a sequence of (eventually) $\epsilon$-feasible solutions.

We start the proof by a series of simple observations.

\begin{lemma}\label{lem:feasibility-kept}
If $y^{(\hat{k})}$ is $\epsilon$-feasible solution to \eqref{eq:SICP} for some $\hat{k}$, then for every $k\geq \hat{k}$, $y^{(k)}$ is also $\epsilon$-feasible.
\end{lemma}
\begin{proof}
If the point $x^{(k)}$ found in Step 2 is  not $\epsilon$-feasible, then a feasibility cut is found, and in Step 5 $y^{(k)}$ is set to be the last $\epsilon$-feasible solution found. Otherwise $y^{(k)} = x^{(k)}$, set in Step 6, is $\epsilon$-feasible.
\end{proof}

\begin{lemma}\label{lem:core}
Suppose that in the beginning of the $k$th iteration we have $\delta\defeq y^{(k-1)}_0 - x^*_0 > 0$, where $x^*$ is an optimal solution of \eqref{eq:SICP}. Then there exists a $\sigma_0 = \sigma_0(\delta) > 0$ (a function of only $\delta$, but not of $k$), such that in the optimal solution of \eqref{eq:master} in Step 2 we have
\[\sigma^{(k)} \geq \sigma_0(\delta) > 0 .\]
\end{lemma}
\begin{proof}
Let $\bar{x}$ be the Slater point whose existence is required by Assumption \ref{ass:1}, and consider the points $x_\lambda = \lambda \bar{x} + (1-\lambda) x^*$ for $\lambda\in(0,1]$. Multiplying the constraints involving $g$ by $1/\eta$, we can assume without loss of generality that $\bar{x}$ satisfies $g(\bar{x},t)\leq -1$ for every $t \in T$. Because of the Slater property of $\bar{x}$ and the feasibility of $x^*$, $x_\lambda$ is a feasible solution of \eqref{eq:master} in every iteration for every $\lambda\in(0,1]$, and it satisfies the inequalities
\begin{align*}
g(x_\lambda, t^{(j)}) + \frac{\lambda}{B}s^{(j)} &\leq \lambda g(\bar{x},t^{(j)}) + (1-\lambda)g(x^*,t^{(j)}) + \lambda\\
& = \lambda (g(\bar{x},t^{(j)}) + 1) + (1-\lambda) g(x^*,t^{(j)})\\
& \leq 0 \quad \text{ for all } j\in J^{(0)}\cup J^{(1)} \cup \cdots,
\end{align*}
using the convexity of $g$ and $s^{(j)}\leq B$ in the first inequality and the Slater condition in the second. In the $k$th iteration, if $y^{(k-1)}_0 - x^*_0 = \delta > 0$, then $x_\lambda$ also satisfies the inequality
\[ y^{(k-1)}_0 - (x_\lambda)_0 = (x^*_0 + \delta) - (\lambda \bar{x}_0 + (1-\lambda) x^*_0) = \delta - \lambda (\bar{x}_0-x^*_0) \geq \delta/2 \]
for every $\lambda>0$ sufficiently small to satisfy $0\leq\lambda(\bar{x}_0-x^*_0)\leq\delta/2$.

Denoting by $\lambda_0$ such a sufficiently small value of $\lambda$, and letting
\[\sigma_0 \defeq \min(\lambda_0/B, \delta/2),\]
we conclude that the pair $(x_{\lambda_0}, \sigma_0)$ is a feasible solution to \eqref{eq:master}, hence the optimal solution to \eqref{eq:master} also satisfies $\sigma^{(k)} \geq \sigma_0 > 0$.
\end{proof}

Our final lemma is required only for the proof of \mbox{Theorem \ref{thm:3}}.

\begin{lemma}\label{lem:sigma-conv-0}
Suppose that \mbox{Algorithm \ref{alg:CONCUSSION}} does \emph{not} terminate. Then the sequence $(\sigma^{(k)})_{k=1,2,\dots}$ decreases monotonically to zero, and the sequence $(y^{(k)}_0)_{k=1,2,\dots}$ is also monotone decreasing.
\end{lemma}

\begin{proof}
For every $k$, $\sigma^{(k)} \geq 0$, because the pair $(x,\sigma) = (x^*,0)$ is a feasible solution in each iteration. From this, and the first inequality of \eqref{eq:master}, the monotonicity of $(y^{(k)}_0)_{k=1,2,\dots}$ follows.

Since $(y^{(k)}_0)_{k=1,2,\dots}$ is monotone decreasing and only inactive cuts are dropped from \eqref{eq:master} in Step 7, the sequence $(\sigma^{(k)})_{k=1,2,\dots}$ is monotone non-increasing. Therefore $(\sigma^{(k)})_{k=1,2,\dots}$ is convergent.

Let us assume (by contradiction) that $\sigma^{(k)} \searrow \sigma_0 > 0$. Then for a sufficiently large $\hat{k}$, $\sigma^{(k)} < \sigma_0\beta$ for every $k\geq\hat{k}$, implying that no cuts are dropped in Step 7 beyond the $\hat{k}$th iteration. Consider the optimal $x^{(j)}$ and $x^{(k)}$ obtained in Step 2 of the $j$th and $k$th iteration, with $k>j\geq \hat{k}$. There are two cases, based on whether a feasibility cut $g(x^{(j)},t^{(j)})>0$ is found in Step 4 of the $j$th iteration or not.

If a feasibility cut is not found in the $j$th iteration, then
\[ x^{(k)}_0 = y^{(k-1)}_0 - \sigma^{(k)} \leq y^{(j)}_0 - \sigma^{(k)} = x^{(j)}_0 - \sigma^{(k)}\]
follows from the first constraint of \eqref{eq:master} in the $k$th iteration, therefore
\[ \|x^{(k)}-x^{(j)}\| \geq \sigma^{(k)} \geq \sigma_0. \]

If a feasibility cut is found in the $j$th iteration, then on one hand we have
\[ g(x^{(j)},t^{(j)}) > 0, \]
and because this cut is not dropped later on, from \eqref{eq:master} in the $k$th iteration we also have
\[ g(x^{(k)},t^{(j)}) + \sigma^{(k)} s^{(j)} \leq 0. \]
%Let $d^{(j)}$ denote the subgradient used in Step 6 of the $j$th iteration, for which $s^{(j)} = h(d^{(j)})$. (If $h$ is a constant, the analysis is analogous to the one below.)
%With this notation, the above two inequalities imply that
From these two inequalities we obtain
\[ 0 \leq \sigma_0 s^{(j)} \leq \sigma^{(k)} s^{(j)} < g(x^{(j)},t^{(j)}) - g(x^{(k)},t^{(j)}) \leq -(d^{(j)})^\mathrm{T}(x^{(k)}-x^{(j)}) \leq \|d^{(j)}\| \cdot \|x^{(k)}-x^{(j)}\| \]
for every $d^{(j)}\in\partial_x g(x^{(j)},t^{(j)})$,
using the convexity of $g(\cdot,t^{(j)})$ and the Cauchy-Schwarz inequality. Note that the strict inequality implies $d^{(j)}\neq 0$. Comparing the left and right-hand sides we obtain
\[ \sigma_0 s^{(j)} / \|d^{(j)}\| < \|x^{(k)}-x^{(j)}\|. \]
From this inequality it follows that as long as the centering parameters $s^{(j)}$ are bounded away from zero and $\|d^{(j)}\|$ is bounded (as assumed), we have a $\sigma_1>0$ independent of $j$ and $k$ satisfying $\sigma_1 < \|x^{(k)}-x^{(j)}\|$.

In summary, regardless of whether we add a feasibility or an optimality cut in iteration $j$, we have that for every $k>j\geq \hat{k}$,
\[ \|x^{(k)}-x^{(j)}\| \geq \min(\sigma_0, \sigma_1) > 0, \]
contradicting the assumption that the sequence $(x^{(k)})_{k=1,2,\dots}$ is bounded, and therefore has an accumulation point.
\end{proof}

With these lemmas, we are ready to prove our main theorems.

\begin{proof}[Proof of Theorem \ref{thm:1}]
Suppose that the algorithm terminates in the $k$th iteration. First assume by contradiction that $y^{(k-1)}$ is not an $\epsilon$-feasible solution to \eqref{eq:SICP}. Then by \mbox{Lemma \ref{lem:feasibility-kept}}, none of the points $y^{(0)},\dots,y^{(k-2)}$ are $\epsilon$-feasible, therefore the upper bound in the first constraint of \eqref{eq:master} is $y^{(k-1)}_0 = U$ (a strict upper bound on the optimum) in every iteration. Hence, by \mbox{Lemma \ref{lem:core}}, $\sigma^{(k)} > 0$, contradicting the assumption that the algorithm terminated.
Therefore $y^{(k-1)}$ is $\epsilon$-feasible.

Now suppose that $y^{(k-1)}$ is $\epsilon$-feasible, but it is not $\epsilon$-optimal, that is, $y^{(k-1)}>x^*_0$. Then by \mbox{Lemma \ref{lem:core}} we have $\sigma^{(k)}>0$ for every $k$, contradicting the assumption that the algorithm terminated.
\end{proof}

\begin{proof}[Proof of Theorem \ref{thm:2}]
Using \mbox{Lemma \ref{lem:feasibility-kept}} it is sufficient to show that at least one $y^{(k)}$ is $\epsilon$-feasible. Suppose otherwise, then no $x^{(k)}$ or $y^{(k)}$ obtained throughout the algorithm is $\epsilon$-feasible. Therefore, the upper bound on the first constraint of \eqref{eq:master} remains $y^{(k-1)} = U$ (a strict upper bound on the optimum) in every iteration. Invoking \mbox{Lemma \ref{lem:core}} we have that $\sigma^{(k)} \geq \sigma_0(U-x^*_0) > 0$, contradicting Lemma \ref{lem:sigma-conv-0}.
\end{proof}

\begin{proof}[Proof of Theorem \ref{thm:3}]
The compactness of the feasible set of \eqref{eq:SICP} implies that if the algorithm does not terminate, then the sequence $(x^{(k)})_{k=1,2,\dots}$ has at least one accumulation point, and so does its subsequence $(y^{(k)})_{k=1,2,\dots}$. From Theorem \ref{thm:2} we also know that this sequence eventually consists entirely of $\epsilon$-feasible points, therefore every accumulation point of the sequence $(y^{(k)})_{k=1,2,\dots}$ is also $\epsilon$-feasible (using that the set of $\epsilon$-feasible solutions is also compact).

Let $\hat{y}$ be one of the accumulation points, and suppose by contradiction that $\hat{y}$ is not $\epsilon$-optimal, that is, $\hat{y}_0 > x^*_0$. Let $\delta = (\hat{y}_0 - x^*_0)/2$, where $x^*$ denotes, as before, an optimal solution to \eqref{eq:SICP}. Using Lemma \ref{lem:feasibility-kept} and the assumption $\delta>0$, there exists a sufficiently large $\hat{k}$ such that for every $k>\hat{k}$, $y^{(k)}$ is an $\epsilon$-feasible solution to \eqref{eq:SICP}, and $y^{(k-1)}_0 \geq x^*_0+\delta$. Invoking Lemma \ref{lem:core} we find that in this case there exists a $\sigma_0 > 0$ such that $\sigma^{(k)} \geq \sigma_0$ for every $k > \hat{k}$, contradicting \mbox{Lemma \ref{lem:sigma-conv-0}}.
\end{proof}

\subsection{Rate of convergence} \label{sec:rateofconvergence}

Recall that throughout the cutting surface algorithm, the sequence $\sigma^{(k)}$ decreases monotonically, and converges to zero (Lemma \ref{lem:sigma-conv-0}). In this section we show that the method converges linearly between feasibility cuts, beyond the first iteration $\hat k$ that satisfies $\sigma^{(\hat k)}<\eta/B$. This matches the rate of convergence of similar cutting plane methods. Interestingly, the analysis can be done in a considerably simpler manner than for the (Kortanek--No) central cutting plane method.

\begin{theorem}\label{thm:rate-of-convergence}
Algorithm \ref{alg:CONCUSSION} converges linearly in objective function value between consecutive feasibility cuts, beyond the first iteration $\hat k$ that satisfies $\sigma^{(\hat k)}<\eta/B$.
\end{theorem}
\begin{proof}
Consider the master problem \eqref{eq:master} and its dual in iteration $k$. Let $\mu_0^{(k)}$ be the optimal value of the dual variable associated with the first constraint, and let $\mu_j^{(k)}$ be the optimal value of the dual variable associated with the constraint corresponding to the index $j\in J^{(k-1)}$.

Without loss of generality it can be assumed that $x_0$, the objective of \eqref{eq:SICP}, is only bounded explicitly from below by constraints in \eqref{eq:SICP}, and therefore the first constraint in the master problem \eqref{eq:master} is always active at the optimum:
\begin{equation}\label{eq:sigmaactive}
\sigma^{(k)} = y_0^{(k-1)}-x_0^{(k)}.
\end{equation}
The dual constraint corresponding to the primal variable $\sigma$ gives
\begin{equation}\label{eq:dual-1}
\mu_0^{(k)} + \sum_{j\in J^{(k-1)}} s^{(j)}\mu_j^{(k)} = 1.
\end{equation}
Using this equation and the optimality of the primal and dual solutions we have that for every $x\in X$ and every $\sigma$,
\begin{subequations}\label{eq:sigma-123}
\begin{align}
\sigma^{(k)} &\geq \sigma - \mu_0^{(k)}(x_0+\sigma-y_0^{(k-1)}) - \sum_{j} \mu_j^{(k)}\big(g(x,t^{(j)}) + \sigma s^{(j)}\big) \label{eq:sigma-1} \\
 &= \mu_0^{(k)}(y_0^{(k-1)}-x_0) - \sum_{j} \mu_j^{(k)}g(x,t^{(j)}) \label{eq:sigma-2}\\
 &\geq \mu_0^{(k)}(y_0^{(k-1)}-x_0). \label{eq:sigma-3}
\end{align}
\end{subequations}

Suppose now that in this iteration the master problem yields an $\epsilon$-feasible solution $x^{(k)}$. Then $y^{(k)} = x^{(k)}$, and Eq.~\eqref{eq:sigmaactive} together with  \eqref{eq:sigma-123} yields
\begin{equation*}
y_0^{(k-1)}-y_0^{(k)} = y_0^{(k-1)}-x_0^{(k)} = \sigma^{(k)} \geq \mu_0^{(k)}(y_0^{(k-1)}-x_0)
\end{equation*}
for every $x\in X$, and specifically for the optimal $x^*$,
\begin{equation*}
y_0^{(k-1)}-y_0^{(k)} = (y_0^{(k-1)} - x^*_0) + (x^*_0 -y_0^{(k)}) \geq \mu_0^{(k)}(y_0^{(k-1)}-x^*_0).
\end{equation*}
Since that $y_0^{(k-1)}$ was not yet optimal, we can divide by $y_0^{(k-1)}-x^*_0 > 0$, which leads to
\begin{equation}\label{eq:rate}
\frac{y_0^{(k)} - x^*_0}{y_0^{(k-1)}-x^*_0} \leq 1- \mu_0^{(k)}.
\end{equation}
From this inequality we immediately have linear convergence in the objective value (between feasibility cuts) provided that we can bound $\mu_0^{(k)}$ away from zero.

To bound $\mu_0^{(k)}$ from below, let us use the notation $M^{(k)} = \sum_{j\in J^{(k-1)}} \mu_j^{(k)}$, and recall \eqref{eq:dual-1} and $s^{(j)}\leq B$. These inequalities imply
\begin{equation}\label{eq:boundmu-1}
\mu_0^{(k)} = 1 - \sum_{j\in J^{(k-1)}} s^{(j)}\mu_j^{(k)} \geq 1 - \sum_{j\in J^{(k-1)}}B \mu_j^{(k)} =  1 - B M^{(k)}.
\end{equation}
Another lower bound can be obtained by substituting the Slater point $\bar{x}$ into \eqref{eq:sigma-2}, and using $g(\bar{x},t^{(j)}) \leq -\eta$:
\begin{equation*}
\sigma^{(k)} \geq \mu_0^{(k)}(y_0^{(k-1)}-\bar{x}_0) - \sum_{j} \mu_j^{(k)}g(\bar{x},t^{(j)}) \geq \mu_0^{(k)}(y_0^{(k-1)}-\bar{x}_0) + \eta M^{(k)} \geq \mu_0^{(k)}(x_0^*-\bar{x}_0) + \eta M^{(k)},
\end{equation*}
which yields
\begin{equation}\label{eq:boundmu-2}
\mu_0^{(k)} \geq \frac{\eta M^{(k)}-\sigma^{(k)}}{\bar{x}_0-x_0^*},
\end{equation}
Taking a linear combination of \eqref{eq:boundmu-1} and \eqref{eq:boundmu-2} with coefficients $\eta>0$ and $B(\bar{x}_0-x^*_0)>0$ eliminates $M^{(k)}$ from the lower bound:
\begin{equation}\label{eq:boundmu-3}
\mu_0^{(k)} \geq \frac{\eta-B\sigma^{(k)}}{\eta+B(\bar{x}_0-x_0^*)}.
\end{equation}
The denumerator on the right is always positive. Since, by assumption, the numerator is bounded away from zero beyond iteration $\hat k$, so is the sequence $\mu_0^{(k)}$, which is what we needed in the inequality \eqref{eq:rate} to complete the proof.
\end{proof}

\deletethis{
\noindent\textbf{Remark.} It is important not to read too much into this rate of convergence result, and specifically into the estimate \eqref{eq:rate}. For example, based on \eqref{eq:rate}, it is compelling to argue that the convergence is superlinear provided that we can enforce $\mu_0^{(k)}=1$, and \eqref{eq:dual-1} suggests that this is achieved if we choose $s^{(j)}=0$ for every $j$, which corresponds to no centering in the algorithm. But the conclusion of Theorem \ref{thm:rate-of-convergence} is moot in this case: with no centering, every iteration of the cutting surface method yields an infeasible point, and no optimality cuts are generated throughout the algorithm. Even with centering, if most iterations end with a feasibility cut, it is hardly meaningful to claim that ``the algorithm has linear convergence''. This is no different for the central cutting plane algorithm, which also does not converge linearly.
}

Cutting methods in general, and our central cutting surface method in particular,
update the best feasible (or in our case, $\epsilon$-feasible) solution found only in those iterations that add an optimality cut to the master problem,
while in the remaining iterations, when a feasibility cut is found, it is the feasible set that gets updated. Therefore, it is difficult to compare the rate of convergence of these methods to the rate of convergence of feasible methods, where the best feasible solution is updated in every iteration, and the rate of convergence of the sequence of objective values can be directly studied. %Nevertheless, we can analyze the rate of decrease in the objective function value between consecutive feasibility cuts.

\section{Applying the central cutting surface algorithm to moment robust optimization}\label{sec:CG}

Our aim in this section is to show that Algorithm \ref{alg:CONCUSSION}, in combination with a randomized column generation
method, is applicable to solving \eqref{eq:DRO} for every objective $h$ that is convex in $x$ (for every $\xi\in\Xi$) as long as
the set $X$ is convex and bounded, and $\mathfrak{P}$ is defined by \eqref{eq:defP}, through lower and upper bounds $(\ell_i,
u_i)$ on some (not necessarily polynomial) moments $\int_\Xi f_i(\xi) P(d\xi)$ of $P$. Bounds can be imposed on moments of
arbitrary order, not only on the first and second moments. The randomized column generation method, presented in Section
\ref{sec:randomized_column_generation}, is an extension of the authors' earlier scenario generation algorithm for stochastic
programming \citep{MehrotraPapp2013}.

We might also consider optimization problems with \emph{robust stochastic constraints}, that is, constraints of the form
\[\mathbb{E}_P[G(x)] \leq 0\quad \forall\,P\in\mathfrak{P} \]
with some convex function $G$. The algorithm presented in this section is applicable verbatim to such problems, but to keep the presentation simple, we consider only the simpler form, \eqref{eq:DRO}. However, we provide a numerical example of our method applied to robust stochastic constraints in Example \ref{ex:3}.

Without loss of generality we shall assume that $f_1$ is the constant one function, and $\ell_1=u_1=1$. We will also use the shorthand $f$ for the vector-valued function $(f_1,\dots,f_N)^\T$.

Our first observation is that while searching for an $\epsilon$-optimal $P$ in \eqref{eq:maxP}, it is sufficient to consider finitely supported distributions.

\begin{theorem}
For every $\epsilon > 0$, the optimization problem \eqref{eq:maxP} has an $\epsilon$-optimal distribution supported on not more than $N+2$ points.
\end{theorem}
\begin{proof}
For every $z\in\real$, the set \[ L_z = \left\{ (v,w)\in\real^{N}\times\real\,\middle|\,\exists P: v = \int_\Xi f(\xi) P(d\xi), w = \int_\Xi h(x,\xi) P(d\xi),\, \ell\leq v\leq u, w\geq z\right\} \]
is an $(N+1)$-dimensional convex set contained in the convex hull of the points
\[\{(f_1(\xi),\dots,f_N(\xi),h(x,\xi))^\T\,|\,\xi\in\Xi\}.\]
Therefore by Carath\'eodory's theorem, as long as there exists a $(v,w)\in L_z$, there also exist $N+2$ points $\xi_1,\dots,\xi_{N+2}$ in $\Xi$ and nonnegative weights $w_1,\dots,w_{N+2}$ satisfying
\[ v = \sum_{k=1}^{N+2} w_k f(\xi_k) \; \text{ and }\; w = \sum_{k=1}^{N+2} w_k h(x,\xi_k). \qedhere\]
\end{proof}

A result of \citep{MehrotraPapp2013} is that whenever the set $\mathfrak{P}$ of distributions is defined as in \eqref{eq:defP}, a
column generation algorithm using randomly sampled columns can be used to find a distribution $P\in\mathfrak{P}$ supported on at
most $N$ points. In other words, a feasible solution to \eqref{eq:maxP} can be found using a randomized column generation
algorithm. In Section \ref{sec:randomized_column_generation} we generalize this result to show that \eqref{eq:maxP} can also be
solved to optimality within a prescribed $\epsilon>0$ accuracy using randomized column generation. The formal description of the
complete algorithm is given in Algorithm \ref{alg:maxP}.  In the remainder of this section we provide a short informal
description and the proof of correctness.

If $\Xi$ is a finite set, then the optimization problem \eqref{eq:maxP} is a linear program whose decision variables are the weights $w_i$ that the distribution $P$ assigns to each point $\xi_i\in\Xi$. In an analogous fashion, \eqref{eq:maxP} in the general case can be written as a semi-infinite linear programming problem with a weight function $w\colon\Xi\mapsto\real_0^+$ as the variable. The corresponding column generation algorithm for the solution of \eqref{eq:maxP} is then the following.

We start with a finite candidate scenario set $\{\xi_1,\dots,\xi_K\}$ that supports a feasible solution. Such points can be
obtained (for instance) using Algorithm 1 in \citep{MehrotraPapp2013}.

At each iteration we take our current candidate scenario set and solve the auxiliary linear program
\begin{equation}\label{eq:maxPaux}
\max_{w\in\real^K} \left\{ \sum_{k=1}^K w_k h(x,\xi_k)\; \middle| \; \ell \leq \sum_{k=1}^K w_k f(\xi_k) \leq u,\, w\geq 0 \right\}
\end{equation}
and its dual problem
\begin{equation}\label{eq:maxPauxdual}
\min_{(p_+,p_-)\in\real^{2N}} \left\{ p_+^\T u - p_-^\T\ell\; \middle| \; (p_+ - p_-)^\T f(\xi_k) \geq h(x,\xi_k)\; (k=1,\dots,K);\; p_+\geq 0, p_-\geq 0 \right\}.
\end{equation}
Note that by construction of the initial node set, the primal problem is always feasible, and since it is also bounded, both the primal and dual optimal solutions exist.

Let $\hat w$ and $(\hat p_+,\hat p_-)$ be the obtained primal and dual optimal solutions; the reduced cost of a point $\xi\in\Xi$ is then
\begin{equation}\label{eq:redcost}
\pi(\xi) \defeq h(x,\xi) - (\hat p_+ - \hat p_-)^\T f(\xi).
\end{equation}
As for every (finite or semi-infinite) linear program, if every $\xi\in\Xi$ has $\pi(\xi)\leq 0$, then the current primal-dual pair is optimal, that is, the discrete probability distribution corresponding to the points $\xi_k$ and weights $\hat w_k$ is an optimal solution to \eqref{eq:maxP}. Moreover, for problem \eqref{eq:maxP} we have the following, stronger, fact.

\begin{theorem}\label{thm:epsopt}
Let $\xi_1,\dots,\xi_K$, $\hat{w}$, and $\pi$ be defined as above, and let $\epsilon\geq0$ be given. If $\pi(\xi) \leq \epsilon$ for every $\xi\in\Xi$,
then the distribution defined by the support points $\xi_1,\dots,\xi_K$ and weights $\hat w_1,\dots, \hat w_K$ is an $\epsilon$-optimal feasible solution to problem \eqref{eq:maxP}.
\end{theorem}
\begin{proof}
The feasibility of the defined distribution follows from the definition of the auxiliary linear program \eqref{eq:maxPaux}, only the $\epsilon$-optimality needs proof.

If the inequality $\pi(\xi) \leq \epsilon$ holds for every $\xi\in\Xi$, then by integration we also have
\begin{equation}\label{eq:epsilon-lemma-eq1} \int_\Xi (\hat p_+ - \hat p_-)^\T f(\xi) P(d\xi) \geq \int_\Xi (h(x,\xi) - \epsilon) P(d\xi) = \int_\Xi h(x,\xi) P(d\xi) - \epsilon
\end{equation}
for every probability distribution $P$. In particular, consider an optimal solution $P^*$ to \eqref{eq:maxP} with $m^* \defeq \int_{\xi\in\Xi} f(\xi)P^*(d\xi)$.
Naturally, $\ell \leq m^* \leq u$, and so we have
\begin{align*}
\sum_{k=1}^K \hat w_k h(x,\xi_k) &= p_+^\T u - p_-^\T\ell \geq (p_+-p_-)^\T m^* =\\
&= \int_\Xi (p_+-p_-)^\T f(\xi) P^*(d\xi)
\geq \int_\Xi h(x,\xi)P^*(d\xi) - \epsilon,
\end{align*}
using strong duality for the primal-dual pair \eqref{eq:maxPaux}-\eqref{eq:maxPauxdual} in the first step, $\ell \leq m^* \leq u$ and the sign constraints on the dual variables in the second step, and inequality \eqref{eq:epsilon-lemma-eq1} in the last step. The inequality between the left- and right-hand sides of the above chain of inequalities is our claim.
\end{proof}

\subsection{Column generation using polynomial optimization}
If we can find a $\xi$ with positive reduced cost, we can add it as $\xi_{K+1}$ to the candidate support set, and recurse. Unfortunately, finding the point $\xi$ with the highest reduced cost, or even deciding whether there exists a $\xi\in\Xi$ with positive reduced cost is NP-hard, even in the case when $\Xi = [0,1]^d$, $h$ is constant zero, and the $f_i$ are the monomials of degree at most two; this follows from the NP-hardness of quadratic optimization over the unit cube.

The only non-trivial special case that is polynomial time solvable is the one where $\pi$ is a polynomial of degree two, and $\Xi$ is an ellipsoid. Then finding $\max_{\xi\in\Xi}\pi(\xi)$ is equivalent to the trust region subproblem of non-linear programming. In other cases, sum-of-squares approximations to polynomial optimization, which lead to tractable semidefinite programming relaxations \citep{Parrilo2003}, could in principle be employed. \cite[Secs.~4--5]{MehrotraPapp2013} summarizes the experience with two existing implementations, GloptiPoly \citep{gloptipoly} and SparsePOP \citep{sparsepop}, in the context of moment matching scenario generation, where a column generation approach similar to the one proposed in this paper leads to pricing problems that are special cases of the ones obtained while solving \eqref{eq:DRO}. In those problems, the largest problems that could be solved using the semidefinite programming approach were three-dimensional problems involving moments up to order $5$.

In order to find a point $\xi$ where $\pi(\xi)>0$ (or prove that such points do not exist) in polynomial time, the global maximum
of $\pi$ need not be found; it would be sufficient to have a polynomial time approximation algorithm with a positive
approximation ratio. However, the only applicable positive result known in this direction is that when $\Xi$ is a simplex, there
exists a polynomial time approximation scheme (PTAS) for every fixed degree \citep{dKLP2006}. Additionally, in low dimensions,
the approximation scheme from \citep{dLHKW2008}, which is fully polynomial time in fixed dimensions, might be useful.

When $\Xi$ is the unit cube and $\pi$ is a multilinear polynomial of degree 2, there is no applicable approximation algorithm unless $NP=ZPP$. When $\Xi$ is the unit sphere and the $\pi$ is a multilinear polynomial of degree 3, there is no applicable approximation algorithm unless $P=NP$. For simple proofs of these results, see the survey \citep{dK2008}; for the best known approximation algorithms for a large number of additional cases we refer to the recent PhD thesis \citep{lithesis2011}.

In conclusion, the available tools for polynomial optimization do not appear to be useful in solving our column generation subproblems. In the next subsection we propose an alternative, practical approach that is also applicable in the non-polynomial setting.

\subsection{Randomized column generation}\label{sec:randomized_column_generation}
Now we show that a column with negative reduced cost can be found with high probability using random sampling. This result does not require the basis functions $f_i$ or the objective $h$ to be polynomials. The randomized column generation method, \mbox{Algorithm \ref{alg:maxP}}, uses the method in \citep{MehrotraPapp2013} in its phase one to generate an initial (feasible, but not necessarily optimal) scenario set and probabilities.

The key observation is that if the functions $h(x,\cdot)$ and $f_i$ are continuously differentiable over the bounded $\Xi$, then the reduced cost function \eqref{eq:redcost} (as a function of $\xi$) also has bounded derivatives. Therefore, sufficiently many independent uniform random samples $\xi_j\in\Xi$ that result in $\pi(\xi_j)\leq 0$ will help us conclude that $\pi(\xi)\leq\epsilon$ for every $\xi\in\Xi$ with high probability. In the following theorem $B(c,r)$ denotes the (Euclidean, $d$-dimensional) ball centered at $c$ with radius $r$.

\begin{theorem}\label{thm:boundp}
Suppose the functions $h(x,\cdot)$ and $f_i$ are continuously differentiable over the closed and bounded $\Xi$, and let $C$ be an upper bound on the gradient of the reduced cost function: $\max_{\xi\in\Xi}\|\nabla \pi(\xi)\| \leq C$. Furthermore, assume that a particular $\tilde\xi\in\Xi$ satisfies $\pi(\tilde\xi) >\epsilon$. Then a uniformly randomly chosen $\xi\in\Xi$ satisfies $\pi(\xi) \leq 0$ with probability \emph{at most} $1-p$, where
\[p = \min_{\xi\in\Xi}\operatorname{vol}(\Xi\cap B(\xi,\epsilon/C))/\operatorname{vol}(\Xi) > 0.\]
In particular, if $\Xi\subseteq\real^d$ is a convex set satisfying $B(c_1,r) \subseteq \Xi \subseteq B(c_2,R)$ with some centers $c_1$ and $c_2$ and radii $r$ and $R$ we have
\[ p > (2\pi(d+2))^{-1/2} \left(\frac{r\epsilon}{2RC}\right)^d. \]
\end{theorem}
\begin{proof}
If $\pi(\tilde\xi) > \epsilon$, then $\pi(\xi) > 0$ for every $\xi$ in its neighborhood $\Xi\cap B(\tilde\xi,\epsilon/C)$. Therefore, the assertion holds with $p(\epsilon,C)=\min_{\xi\in\Xi}\operatorname{vol}(\Xi\cap B(\xi,\epsilon/C))/\operatorname{vol}(\Xi)$. This minimum exists, because $\Xi$ is closed and bounded; and it is positive, because the intersection is a non-empty closed convex set for every center $\xi$.

To obtain the lower bound on $p$, we need to bound from below the volume of the intersection $\Xi\cap B(\xi,\epsilon/C)$. Consider the right circular cone with apex $\xi$ whose base is the $(d-1)$-dimensional intersection of $B(c_1,r)$ and the hyperplane orthogonal to the line connecting $c_1$ and $\xi$. This cone is contained within $\Xi$, and all of its points are at distance $2R$ or less from $\xi$. Shrinking this cone with respect to the center $\xi$ with ratio $\epsilon/(2RC)$ yields a cone contained in $\Xi\cap B(\xi,\epsilon/C)$. Using the volume of this cone as a lower bound on $\operatorname{vol}(\Xi\cap B(\xi,\epsilon/C))$ and the notation $V_d(r)$ for the volume of the $d$-dimensional ball of radius $r$, we get
\begin{align*}\frac{\operatorname{vol}(\Xi\cap B(\xi,\epsilon/C))}{\operatorname{vol}(\Xi)} &\geq \frac{(d+1)^{-1}V_{d-1}(r)r}{V_d(R)}\left(\frac{\epsilon}{2RC}\right)^d =
\frac{\pi^{(d-1)/2}\Gamma((d+2)/2)}{(d+1)\pi^{d/2}\Gamma((d+1)/2)}\left(\frac{\epsilon r}{2RC}\right)^d\\
&= \pi^{-1/2}\frac{\Gamma((d+2)/2)}{2\Gamma((d+3)/2)}\left(\frac{\epsilon r}{2RC}\right)^d > \pi^{-1/2}\cdot(2d+4)^{-1/2}\left(\frac{\epsilon r}{2RC}\right)^d,
\end{align*}
with some lengthy (but straightforward) arithmetic in the last inequality, using the log-convexity of the gamma function.
\end{proof}

Theorem \ref{thm:boundp}, along with Theorem \ref{thm:epsopt}, allows us to bound the number of uniform random samples $\xi\in\Xi$ we need to draw to be able to conclude with a fixed low error probability, that the optimal solution of \eqref{eq:maxPaux} is an $\epsilon$-optimal solution to \eqref{eq:maxP}. This is an explicit, although very conservative, bound: with $\hat p$ given in each iteration, and known global bounds on the gradients of $h$ and the components of $f$, an upper bound $C$ on $\|\nabla \pi(\cdot)\|$ can be easily computed in every iteration. (A global bound, valid in every iteration, can also be obtained whenever the dual variables $\hat p$ can be bounded a priori.) This provides the (probabilistic) stopping criterion for the column generation for Algorithm \ref{alg:maxP}. Note that the $\epsilon$ used in Theorems \ref{thm:epsopt} and \ref{thm:boundp} is the
same $\epsilon$ used in the termination criteria for solving \eqref{eq:DRO-SICP} using Algorithm \ref{alg:maxP}.

In order to use Theorem \ref{thm:boundp}, we need an efficient algorithm to sample uniformly from the set $\Xi$. This is obvious if $\Xi$ has a very simple geometry, for instance, when $\Xi$ is a $d$-dimensional rectangular box, simplex, or ellipsoid. Uniform random samples can also be generated efficiently from general polyhedral sets given by their facet-defining inequalities and also from convex sets, using random walks with polynomial mixing times. See, for example, the survey \citep{Vempala2005} for uniform sampling methods in polyhedra. A strongly polynomial method for polyhedra was found more recently in \citep{KannanNarayanan2012}; a weakly polynomial method for convex sets appears in \citep{LovaszVempala2006}. \citep{HuangMehrotra2013} also gives a detailed and up-to-date list of references on uniform sampling on convex sets.

We can now conclude that the semi-infinite convex program formulation of \eqref{eq:DRO} can be solved using Algorithm \ref{alg:CONCUSSION}, with Algorithm \ref{alg:maxP} and an efficient uniform sampling method serving as a probabilistic version of the oracle required by Assumption \ref{ass:2}.

\begin{figure}[ht]
\begin{mdframed}[innerleftmargin=2em]
\begin{algorithm}[Randomized column generation method to solve \eqref{eq:maxP}-\eqref{eq:defP}]\label{alg:maxP}
\hfill\par
\begin{enumerate}
\item[] Parameters: $M$, the maximum number of random samples per iteration. (See the text for details on choosing this parameter.)
\item[Step 1.] Find a finitely supported feasible distribution to \eqref{eq:maxP} using Algorithm 1 in \citep{MehrotraPapp2013}. Let $S=\{\xi_1,\dots,x_K\}$ be its support.
\item[Step 2.] Solve the primal-dual pair \eqref{eq:maxPaux}-\eqref{eq:maxPauxdual} for the optimal $\hat w$, $\hat p_+$, and $\hat p_-$.
\item[Step 3.] Sample uniform random points $\xi\in\Xi$ until one with positive reduced cost $h(x,\xi) - (\hat p_+ - \hat p_-)^\T f(\xi)$ is found or the maximum number of samples $M$ is reached.
\item[Step 4.] If in the previous step a $\xi$ with positive reduced cost was found, add it to $S$, increase $K$, and return to Step 2. Otherwise stop.
\end{enumerate}
\end{algorithm}
\end{mdframed}
\end{figure}

\section{Numerical results}\label{sec:numerical}

\subsection{Semi-infinite convex optimization problems}
Most standard benchmark problems in the semi-infinite programming literature are linear. When the problem \eqref{eq:SICP} is linear, Algorithm \ref{alg:CONCUSSION} reduces to the central cutting plane algorithm (except for our more general centering); therefore we only consider convex non-linear test problems from the literature. The results in this section are based on an implementation of the central cutting plane and central cutting surface algorithms using the AMPL modeling language and the MOSEK and CPLEX convex optimization software. The comparison between the algorithms is based solely on the number of iterations. The running times for all the examples were comparable in all instances, and were less than 5 seconds on a standard desktop computer, except for the $20$- and $40$-dimensional instances of Example \ref{ex:4e}, where the central cutting plane method needed considerably more time to converge than Algorithm \ref{alg:CONCUSSION}.

We start by an illustrative example comparing the central cutting plane algorithm of \citet{KortanekNo1993} and our central cutting surface algorithm.

\begin{example}[\citealt{TichatschkeNebeling1988}]
\begin{equation}
\begin{split}
\textrm{minimize}   &\quad (x_1-2)^2 + (x_2-0.2)^2\\
\textrm{subject to} &\quad (5\sin(\pi\sqrt{t})/(1+t^2))x_1^2 - x_2 \leq 0 \quad\forall t\in[0,1]\\
                    &\quad x_1 \in [-1,1], x_2 \in [0,0.2].
\end{split}\label{eq:Ex1}
\end{equation}\label{ex:1}
\end{example}
The example is originally from \citep{TichatschkeNebeling1988}, and it is used frequently in the literature since. (In the original paper the problem appears with $t\in[0,8]$ in place of $t\in[0,1]$ in the infinite constraint set. We suspect that this is a typographic error: not only is that a less natural choice, but it also renders the problem non-convex.)

The optimal solution is $x=(0.20523677,0.2)$. This problem is particularly simple, as only one cut is active at the optimal solution (it corresponds to $\hat{t}\approx0.2134$), and this is also the most violated inequality for every $x$.

We initialized both algorithms with the trivial upper bound $5$ on the minimum, corresponding to the feasible solution $(0,0)$. Tbl.~\ref{tbl:ex1} shows the progress of the two algorithms (using constant centering parameter $s^{(k)}=1$ in both algorithms), demonstrating that both algorithms have an empirical linear rate of convergence. The central cutting plane method generates more cuts (including multiple feasibility cuts at the point $\hat{t}$). On the other hand, the cutting surface algorithm generates only a single cut at $\hat{t}$ in the first iteration, and then proceeds by iterating through central feasible solutions until optimality is established.

\begin{table}
\centering
\begin{tabular}{ccccccc}
\toprule
         & \multicolumn{3}{c}{cutting surface} & \multicolumn{3}{c}{cutting plane}\\
\midrule
\multirow{2}{*}{$\sigma$} & feasibility & optimality & relative & feasibility & optimality & relative\\
 & cuts & cuts & error & cuts & cuts & error\\
\midrule
%$10^{-4}$& 1 & 23 & $10^{-4.283}$ & 4 & \phantom{0}56 & $10^{-3.976}$ \\
%$10^{-5}$& 1 & 29 & $10^{-5.413}$ & 4 & \phantom{0}71 & $10^{-4.610}$ \\
%$10^{-6}$& 1 & 34 & $10^{-6.356}$ & 4 & \phantom{0}90 & $10^{-5.664}$ \\
%$10^{-7}$& 1 & 39 & $10^{-7.304}$ & 4 & 109 & $10^{-6.673}$ \\
%$10^{-8}$& 1 & 44 & $10^{-8.304}$ & 4 & 128 & $10^{-7.776}$ \\
$10^{-4}$& 1 & 23 & $10^{-4.283}$ & 7 & 24 & $10^{-4.856}$ \\
$10^{-5}$& 1 & 29 & $10^{-5.413}$ & 7 & 29 & $10^{-5.083}$ \\
$10^{-6}$& 1 & 34 & $10^{-6.356}$ & 7 & 37 & $10^{-6.157}$ \\
$10^{-7}$& 1 & 39 & $10^{-7.304}$ & 8 & 43 & $10^{-7.174}$ \\
\bottomrule
\end{tabular}
\caption{Comparison of the central cutting surface and central cutting plane algorithms in Example \ref{ex:1}, with centering parameters $s^{(k)}=1$. $\sigma$ for the cutting plane algorithm is an identical measure of the distance from the optimal solutions as in Algorithm \ref{alg:CONCUSSION}; both algorithms were terminated upon reaching $\sigma < 10^{-7}$. The relative error columns show the relative error from the true optimal objective function value. Both algorithms clearly exhibit linear convergence, but the cutting surface algorithm needs only a single cut and fewer iterations.}\label{tbl:ex1}
\end{table}

%\begin{example}[\citealt{BGB2005}]
%\begin{equation}
%\begin{split}
%\textrm{minimize}   &\quad x_0\\
%\textrm{subject to} &\quad -x_0 + x_1^2(2t^2+1)-t^4 \leq 0 \quad\forall t\in[-1,1]
%\end{split}\label{eq:Ex2}
%\end{equation}
%\end{example}
%This is in an even simpler, effectively single-variable problem with a one-dimensional constraint index set. The optimal solution is $x=(0,0)$, which the cutting plane method finds in 19 iterations, generating 17 cuts in the process, all corresponding to same $t=0$. As in the previous example, the cutting surface algorithm generates only one cut for $t=0$, in the first iteration, and generates no further cuts throughout the algorithm.

\begin{example}[Smallest enclosing sphere]\label{ex:2}
The classic smallest enclosing ball and the smallest enclosing ellipsoid
problems ask for the sphere or ellipsoid of minimum volume that contains a finite set of given points. Both of them admit well-known second order cone programming and semidefinite programming formulations. A natural generalization is the following: given a closed parametric surface $p(t)$, $t\in T$ (with some given $T\subseteq\real^n)$, find the sphere or ellipsoid of minimum volume that contains all points of the surface. These problems also have a semi-infinite convex programming formulation. The smallest enclosing sphere, centered at $x$ with radius $r$, is given by the optimal solution of
\[
\text{minimize}\; r\quad \text{subject to}\; \|x-p(t)\| \leq r,\;\; \forall\,t \in T,
\]
whereas the smallest enclosing ellipsoid is determined by
\[
\text{maximize}\; (\det A)^{(1/n)} \quad \text{subject to}\; A \succcurlyeq 0 \;\text{ and }\; \|x-Ap(t)\| \leq 1,\;\;\forall\,t \in T.
\]
In the latter formulation $A\succcurlyeq 0$ denotes that the matrix $A$ is positive semidefinite. The objective function $\log(\det(A))$ could also be used in place of $\det(A)^{1/n}$; the two formulations are equivalent.
\end{example}

It was shown in \citep{PappAlizadeh2011} that these problems also admit a semidefinite programming (SDP) formulation whenever every component of $p$ is a polynomial or a trigonometric polynomial of a single variable. This yields a polynomial time solution, but the formulation might suffer from ill-conditioning whenever the degrees of the polynomials (or trigonometric polynomials) involved is too large. Additionally, the sum-of-squares representations of nonnegative (trigonometric) polynomials that the SDP formulation hinges on do not generalize to multivariate polynomials. The central surface cutting algorithm does not have comparable running time guarantees to those of semidefinite programming algorithms, but it is applicable in a more general setting (including multi-dimensional index sets $T$ corresponding to multivariate polynomials), and does not suffer from ill-conditioning.

We give two examples of different complexity. First, consider the two-dimensional parametric curve
\begin{equation}
p(t) = (c \cos(t)-\cos(ct), c \sin(t)-\sin(ct)),\quad c=4.5,\, t\in[0,4\pi].
\label{eq:seb_1}
\end{equation} This symmetric curve has a smallest enclosing circle centered at the origin, touching the curve at 7 points. (Fig.~\ref{seb_1_a}.)

\begin{figure}
\centering
\subfigure[]{
  \includegraphics[height=175pt]{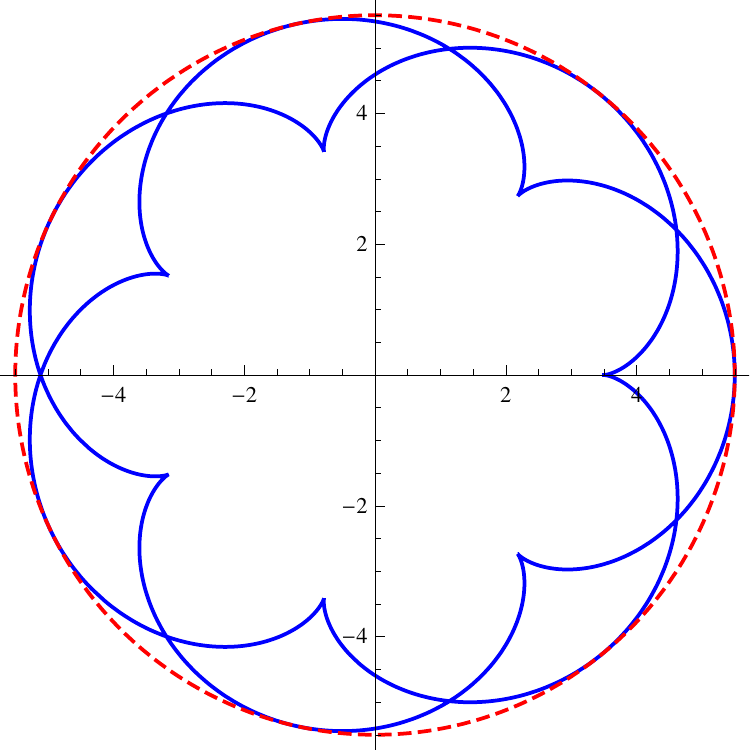}\label{seb_1_a}
}
\subfigure[]{
  \includegraphics[height=175pt]{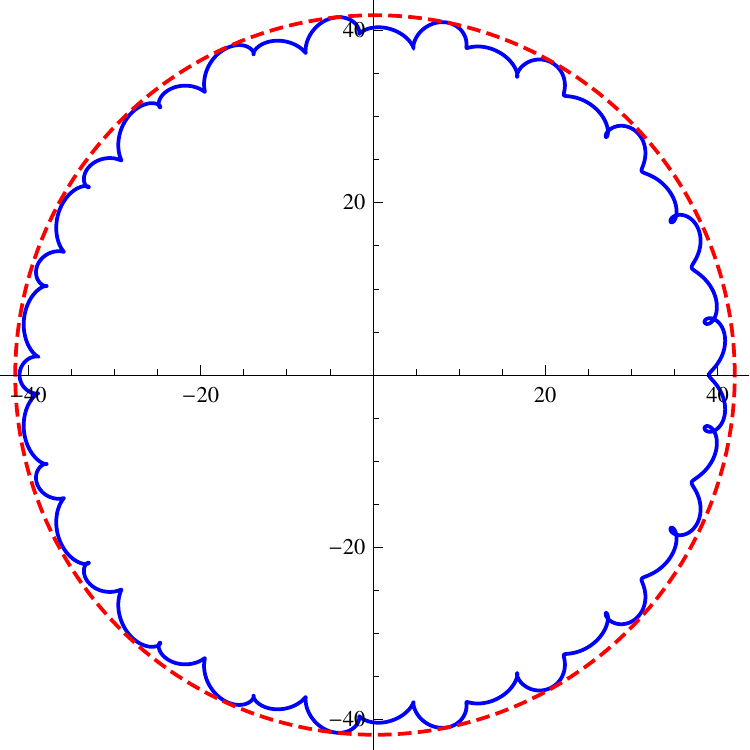}\label{seb_1_b}
}
\caption{The parametric curves \eqref{eq:seb_1} and \eqref{eq:seb_2}, and their smallest enclosing circles.}\label{fig:seb_1}
\end{figure}

Tbl.~\ref{tbl:seb_1} shows the rate of convergence of the two algorithms (using constant centering parameter $s^{(k)}=1$ in both algorithms). The initial upper bound on the minimum was set to $2(c+1)^2$, obtained by a simple term-by-term bound on the objective. In this example, the number of optimality cuts is approximately the same for the two algorithms, but there is a difference in the number of feasibility cuts, and consequently in the total number of iterations.

\begin{table}[!ht]
\centering
\begin{tabular}{ccccccc}
\toprule
         & \multicolumn{3}{c}{cutting surface} & \multicolumn{3}{c}{cutting plane}\\
\midrule
\multirow{2}{*}{$\sigma$} & feasibility & optimality & relative & feasibility & optimality & relative\\
 & cuts & cuts & error & cuts & cuts & error\\
\midrule
%$10^{-4}$& 6 & 17 & $10^{-5.543}$ & 8 & 105 & $10^{-4.441}$ \\
%$10^{-5}$& 6 & 21 & $10^{-7.011}$ & 8 & 132 & $10^{-5.464}$ \\
%$10^{-6}$& 8 & 23 & $10^{-7.567}$ & 8 & 159 & $10^{-6.562}$ \\
%$10^{-7}$& 8 & 26 & $10^{-8.180}$ & 9 & 182 & $10^{-7.428}$ \\
%$10^{-8}$& 8 & 30 & $10^{-8.440}$ & 9 & 208 & $10^{-8.180}$ \\
$10^{-4}$& 6 & 16 & $10^{-5.267}$ & 12 & 16 & $10^{-5.705}$ \\
$10^{-5}$& 6 & 20 & $10^{-6.845}$ & 13 & 18 & $<10^{-10}$ \\
$10^{-6}$& 6 & 23 & $<10^{-10}$ & 14 & 22 & $<10^{-10}$ \\
$10^{-7}$& 6 & 26 & $<10^{-10}$ & 14 & 27 & $<10^{-10}$ \\
$10^{-8}$& 6 & 28 & $<10^{-10}$ & 14 & 28 & $<10^{-10}$ \\
\bottomrule
\end{tabular}
\caption{Comparison of the central cutting surface and central cutting plane algorithms on the first curve of Example \ref{ex:2}, with centering parameters $s^{(k)}=1$. $\sigma$ for the cutting plane algorithm is an identical measure of the distance from the optimal solutions as in Algorithm \ref{alg:CONCUSSION}; both algorithms were terminated upon reaching $\sigma < 10^{-8}$.}\label{tbl:seb_1}
\end{table}

Now consider an asymmetric, high-degree variant of the previous problem, depicted on Fig.~\ref{seb_1_b}:
\begin{equation}
p(t) = (c \cos(t)-\cos(ct), \sin(20t) + c \sin(t)-\sin(ct)),\quad c=40,\, t\in[0,2\pi].
\label{eq:seb_2}
\end{equation}
The center is no longer at the origin, and a closed form description of the circle is difficult to obtain. The semidefinite programming based solution of \citep{PappAlizadeh2011} is theoretically possible, but practically not viable, owing to the high degree of the trigonometric polynomials involved. Tbl.~\ref{tbl:seb_2} shows the rate of convergence of the two algorithms (using constant centering parameter $s^{(k)}=1$ in the cutting surface algorithm).
%The difference between the cutting surface and cutting plane algorithms is now even more pronounced than in the previous example. The cutting plane algorithm requires hundreds of iterations for each additional digit of precision, whereas the cutting surface algorithm has essentially the same performance as in the previous, low-degree, example.

\begin{table}[!ht]
\centering
\begin{tabular}{ccccccc}
\toprule
         & \multicolumn{3}{c}{cutting surface} & \multicolumn{3}{c}{cutting plane}\\
\midrule
\multirow{2}{*}{$\sigma$} & feasibility & optimality & relative & feasibility & optimality & relative\\
 & cuts & cuts & error & cuts & cuts & error\\
\midrule
%$10^{-4}$& 6 & 24 & $10^{-7.516}$ & \phantom{0}9 & \phantom{0}856 & $10^{-5.321}$ \\
%$10^{-5}$& 6 & 27 & $10^{-8.456}$ & \phantom{0}9 & 1050 & $10^{-6.326}$ \\
%$10^{-6}$& 6 & 30 & $<10^{-10}$   & \phantom{0}9 & 1243 & $10^{-7.343}$ \\
%$10^{-7}$& 6 & 34 & $<10^{-10}$   & \phantom{0}9 & 1437 & $10^{-8.629}$ \\
%$10^{-8}$& 6 & 37 & $<10^{-10}$   & 11 &1580 & $<10^{-10}$ \\
$10^{-4}$& 6 & 23 & $10^{-7.517}$ & 15 & 21 & $10^{-5.321}$ \\
$10^{-5}$& 6 & 26 & $10^{-8.463}$ & 15 & 24 & $10^{-8.463}$ \\
$10^{-6}$& 6 & 29 & $<10^{-10}$   & 17 & 27 & $<10^{-10}$ \\
$10^{-7}$& 6 & 32 & $<10^{-10}$   & 17 & 30 & $<10^{-10}$ \\
$10^{-8}$& 7 & 35 & $<10^{-10}$   & 17 & 34 & $<10^{-10}$ \\
\bottomrule
\end{tabular}
\caption{Comparison of the central cutting surface and central cutting plane algorithms on the second curve of Example \ref{ex:2}, with centering parameters $s^{(k)}=1$. $\sigma$ for the cutting plane algorithm is an identical measure of the distance from the optimal solutions as in Algorithm \ref{alg:CONCUSSION}; both algorithms were terminated upon reaching $\sigma < 10^{-8}$. The relative error columns show the relative error from the true optimal objective function value.}\label{tbl:seb_2}
\end{table}

In our next example we consider a generalization of the above problems, a problem with second order cone constraints of dimension
higher than two, and investigate the hypothesis that cutting surfaces may be particularly advantageous in higher dimensions, when
a polyhedral approximation of the feasible set is expensive to build.

\begin{example}\label{ex:4e}
Consider the SICP
\[ \min_{x\in[-1,1]^n} \max_{t\in[0,1]} \sum_{i=1}^n (ix_i - i/n - \sin(2\pi t + i))^2.  \]
It is easy to see that the optimal solution is $x = (1/n,1/n,\dots,1/n)$.
\end{example}

The initial upper bound $U=4n$ on the minimum can be obtained by taking a term-by-term upper bound of the objective at $x=0$. We used this bound to initialize the central cutting surface and central cutting plane algorithms. As in the above examples, we used the centering parameter $s^{(k)}=1$ in both algorithms.

Tbl.~\ref{tbl:ex4e_s1} shows the number of feasibility cuts and the number of optimality cuts  necessary until the stopping
condition $\sigma < 10^{-6}$ is satisfied for different values of $n$.

\begin{table}[!ht]
\centering
\begin{tabular}{lcccccc}
\toprule
         & $n=$  &  5   & 10    & 20    & 40 \\
\midrule
cutting surface && 13+19 & 16+17 & 15+19  & 15+22  \\
cutting plane   && 93+14 & 290+15 & 1179+15 & $>$10000  \\
\bottomrule
\end{tabular}
\caption{Comparison of the central cutting surface and central cutting plane algorithms on Example \ref{ex:4e}, for different values of $n$ (the number of decision variables). Each entry in the table is in the format \emph{the number of feasibility cuts + the number of optimality cuts}, obtained with the centering parameter $s^{(k)}=1$. Both algorithms were terminated upon reaching $\sigma < 10^{-6}$ or after 10000 cuts.}\label{tbl:ex4e_s1}
\end{table}

It is clear that in this example the number of feasibility cuts (and the total number of cuts) in the cutting plane algorithm grows much more rapidly with dimension than in the cutting surface algorithm. This is consistent with the fact that, unless strong centering is applied, a good polyhedral approximation (for cutting planes) or conic approximation (for cutting surfaces) of the feasible set needs to be built, which requires considerably more planar cuts than surface cuts. In the next section we consider the effect of centering further.

\subsubsection{The effect of the centering parameter}

The fact in Examples \ref{ex:1} and \ref{ex:2} most generated cuts are optimality cuts, not feasibility cuts, suggests that our default setting of the centering parameter, $s^{(k)}=1$ in each iteration $k$, might not be optimal. At the other extreme, $s^{(k)} = 0$ is expected to yield infeasible solutions in all iterations but the last. Another natural choice for the centering parameter, as discussed in Section \ref{sec:algorithm}, is the gradient of the norm of the violated inequality, which is suggested by Kortanek and No in their central cutting plane algorithm. Finally, our convergence proof shows that one can also use a constant fraction of this gradient norm. Example \ref{ex:4e} also suggests that the centering parameter that keeps a balance between feasibility and optimality cuts might be different for the two algorithms, and that centering might be less important for cutting surfaces than for cutting planes (which must avoid building expensive polyhedral approximations of the feasible set around points that are far from the optimum). In this section we further examine (empirically) the effect of the centering parameter.

The smallest examples above solved by the cutting surface algorithm with no centering in only two iterations; for instance, in Example \ref{ex:1}, the cutting surface algorithm generates one feasibility cut (at the same point $\hat{t}$ as the cutting surface algorithm with centering), and then one optimality cut, after which the optimality is proven.

For a non-trivial example, consider the second instance of the smallest enclosing sphere problems in Example \ref{ex:2}, with the parametric curve defined in \eqref{eq:seb_2}, and solve again the corresponding SICP problem using Algorithm \ref{alg:CONCUSSION}, as well as the central cutting plane algorithm of Kortanek and No, using different constant centering parameters $s^{(k)}$. Tbls.~\ref{tbl:varyingcentrality-1} and \ref{tbl:varyingcentrality-2} show the number of feasibility and optimality cuts for different values of this parameter. (The stopping criterion was $\sigma<10^{-8}$.)

\begin{table}[h]
\centering
\begin{tabular}{lrrrrrrrrrrrr}
\toprule
 $s^{(k)}$ & $10^{-9}$ & $10^{-7}$ & $10^{-5}$ & $10^{-3}$ & $10^{-2}$ & $10^{-1}$ & $1.$ & $10^1$ & $10^2$  \\
\midrule
%feasibility cuts & 7 & 7 & 7 & 7 &  7 &  6 &   6 &    6\\
%optimality cuts  & 2 & 3 & 4 & 6 & 12 & 36 & 237 & 2038\\
cutting surfaces\\
feasibility cuts & 9 & 8 & 7 & 7 & 7 &  7 &  7 &   9 &   10 \\
optimality cuts  & 2 & 2 & 3 & 4 & 6 & 11 & 35 & 190 & 1496 \\
\midrule
cutting planes\\
feasibility cuts & 18 & 18 & 16 & 16 & 16 & 17 & 17 &  23 &   27 \\
optimality cuts  &  2 &  2 &  3 &  4 &  6 & 11 & 34 & 195 & 1827 \\
\bottomrule
\end{tabular}
\caption{The effect of centering on the number of cuts in the central cutting surface and central cutting plane algorithms using a constant centering parameter.}\label{tbl:varyingcentrality-1}
\end{table}

\begin{table}[h]
\centering
\begin{tabular}{lrrrrrrrrr}
\toprule
 $s^{(k)}/\|\nabla g(x^{(k)},t^{(k)})\|$ & $10^{-9}$ & $10^{-7}$ & $10^{-5}$ & $10^{-3}$ & $10^{-2}$ & $10^{-1}$ & $1.$ \\
\midrule
%feasibility cuts & 7 & 7 & 7 & 7 &  7 &  6 &   6 &    6\\
%optimality cuts  & 2 & 3 & 4 & 6 & 12 & 36 & 237 & 2038\\
cutting surfaces\\
feasibility cuts &  7  &  7 &  7 &  7 &  7 &   9 &   10 \\
optimality cuts  &  2  &  3 &  4 & 11 & 33 & 183 & 1379 \\
\midrule
cutting planes\\
feasibility cuts & 18 & 18 & 16 & 16 & 16 &  26 & {\it   22} \\
optimality cuts  &  2 &  3 &  6 & 10 & 30 & 155 & {\it 1524} \\
\bottomrule
\end{tabular}
\caption{The effect of centering on the number of cuts in the central cutting surface and central cutting plane algorithms using a constant fraction of the gradient norm as centering parameter. The italic numbers in the last column indicate the original central cutting plane algorithm as proposed in \citep{KortanekNo1993}. Even the central cutting plane algorithm benefits considerably from adjusting the centering parameter.}\label{tbl:varyingcentrality-2}
\end{table}

It is interesting to note that the original central cutting plane algorithm, as proposed in \citep{KortanekNo1993}, which uses the gradient norm as the centering parameter, performs particularly poorly in this example. (See the last column of Table \ref{tbl:varyingcentrality-2}.) Even this method benefits from adjusting (in this case, lowering) the centering parameter.

Now let us consider Example \ref{ex:4e}, and solve it again with choices for of the centering parameter. Tbl.~\ref{tbl:ex4e_s1} in the previous section shows the results for $s^{(k)}=1$. Tbl.~\ref{tbl:ex4e_s0} shows what happens with no centering, while Tbls.~\ref{tbl:ex4e_s1e-2g}--\ref{tbl:ex4e_s1g} show results with centering using different fractions of the gradient norm.

\begin{table}[!ht]
\centering
\begin{tabular}{lcccccc}
\toprule
         & $n=$  &  5   & 10    & 20    & 40 \\
\midrule
cutting surface && 14+1 & 17+1  & 22+1  & 22+1 \\
cutting plane   && 94+1 & 402+1 & 4972+1 & $>$10000 \\
\bottomrule
\end{tabular}
\caption{Results from Example \ref{ex:4e} using $s^{(k)}=0$ (no centering). Each entry in the table shows the number of feasibility cuts + the number of optimality cuts. The stopping criterion $\sigma < 10^{-6}$.}\label{tbl:ex4e_s0}
\end{table}

\begin{table}[!ht]
\centering
\begin{tabular}{lcccccc}
\toprule
         & $n=$  &  5   & 10    & 20    & 40 \\
\midrule
cutting surface && 13+8 & 15+11 & 16+20  & 11+47   \\
cutting plane   && 87+6 & 304+9 & 1139+16 & 4510+34  \\
\bottomrule
\end{tabular}
\caption{Results from Example \ref{ex:4e} using $s^{(k)}=10^{-2}\|\nabla\|$. Each entry in the table shows the number of feasibility cuts + the number of optimality cuts. The stopping criterion $\sigma < 10^{-6}$.}\label{tbl:ex4e_s1e-2g}
\end{table}

\begin{table}[!ht]
\centering
\begin{tabular}{lcccccc}
\toprule
         & $n=$  &  5   & 10    & 20    & 40 \\
\midrule
cutting surface && 15+24 & 14+48 & 13+123 & 10+369   \\
cutting plane   && 99+18 & 279+36 & 922+87 & 3483+232  \\
\bottomrule
\end{tabular}
\caption{Results from Example \ref{ex:4e} using $s^{(k)}=10^{-1}\|\nabla\|$. Each entry in the table shows the number of feasibility cuts + the number of optimality cuts. The stopping criterion $\sigma < 10^{-6}$.}\label{tbl:ex4e_s1e-1g}
\end{table}

\begin{table}[!ht]
\centering
\begin{tabular}{lcccccc}
\toprule
         & $n=$  &  5   & 10    & 20    & 40 \\
\midrule
cutting surface && 12+175 &  12+383 & 11+886 & 8+3115 \\
cutting plane   && 92+102 & 250+247 & 823+705 & 2990+1971  \\
\bottomrule
\end{tabular}
\caption{Results from Example \ref{ex:4e} using $s^{(k)}=\|\nabla\|$. Each entry in the table shows the number of feasibility cuts + the number of optimality cuts.}\label{tbl:ex4e_s1g}
\end{table}

The results exhibit some interesting phenomena. First, the cutting surface algorithm benefits less from strong centering than cutting planes, although it does benefit from some centering. It is also apparent that cutting planes require higher values for the centering parameter before the intermediate solutions become central (feasible). In the extreme case, with no centering (Table \ref{tbl:ex4e_s0}), both methods generate infeasible points throughout the algorithm, until an $\epsilon$-feasible point is found. In this case, the algorithm ends with an optimality cut in the last iteration.

The results also indicate that the central cutting plane algorithm is more sensitive to the choice of the centering parameter than the cutting surface algorithm.

Finally, it appears that in the high-dimensional instances cutting planes cannot compete with even the plain, uncentered, cutting surfaces, regardless of the type of centering used in the cutting plane method. This is explained by the fact that the high-dimensional convex feasible set cannot be approximated well by a small number of planar cuts. This is one setting where we expect the cutting surface method to be superior to cutting planes in general.

\subsection{Robust, distributionally robust, and stochastic optimization}

To illustrate the use of the central cutting surface algorithm in moment robust optimization (Section \ref{sec:DRO}), we return to Example \ref{ex:1}, and turn it into a problem with robust stochastic constraints:

\begin{example}\label{ex:3}
\begin{equation}
\begin{split}
\textrm{minimize}   &\quad (x_1-2)^2 + (x_2-0.2)^2\\
\textrm{subject to} &\quad \mathbb{E}_P[(5\sin(\pi\sqrt{\xi})/(1+\xi^2))x_1^2 - x_2] \leq 0 \quad\forall P\in\mathfrak{P}_m\\
                    &\quad x_1 \in [-1,1], x_2 \in [0,0.2],
\end{split}\label{eq:ex3}
\end{equation}
where $\mathfrak{P}_m$ is a set of probability distributions supported on $\Xi = [0,1]$ with prescribed polynomial moments up to order $m$:
\[ \mathfrak{P}_m \defeq \{P\,|\, \mathbb{E}_P [\xi^i] = 1/(i+1), i=0,\dots,m \}.\]
\end{example}

Setting $m=0$ in the above formulation gives the classic robust optimization version of the problem, which is equivalent to the original Example \ref{ex:1}.

At the other extreme, $\mathfrak{P}_\infty$ contains only the uniform distribution supported on $[0,1]$. Therefore, solving \eqref{eq:ex3} for $m=\infty$ amounts to solving a stochastic programming problem with a continuous scenario set. (Recall \mbox{Theorem \ref{thm:DRO-to-SP}}.) We solved a highly accurate deterministic approximation of this problem by replacing the continuous scenario set with a discrete one, corresponding to the 256-point Gaussian rule for numerical integration; this case, therefore, does not require the solution of a SICP.

The solutions to problem \eqref{eq:ex3} for increasing values of $m$ correspond to less and less conservative (or risk-averse) solutions. It is instructive to see how the solutions of these problems evolve as we impose more and more moment constraints, moving from the robust optimization solution to the stochastic programming solution. In particular, this simple problem illustrates the value of moment information beyond the first and second moments. Interestingly, at the same time, there is no increase in the number of cuts necessary to find the optimum.

The results are summarized in Tbl.~\ref{tbl:ex3}. Note the rather large difference between the optimal values of $x_1$ and the objective function upon the addition of the first few moment constraints.

% toyexample1.m
\begin{table}[!h]
\centering
\begin{tabular}{cccccc}
\toprule
%         & \multicolumn{2}{c}{cutting surface} & \multicolumn{2}{c}{cutting plane} \\
$m$ & optimality cuts & feasibility cuts & $x_1$ & $x_2$ & $z$ \\
\midrule
%$0$ & 7 & 3 & $0.205268315045211$ & $0.2$ & $3.22106181266803$ \\
%\midrule
%$1$ & 8 & 3 & $0.246541528538983$ & $0.2$ & $3.07461660114852$ \\
%$2$ & 7 & 2 & $0.247120107309199$ & $0.2$ & $3.07258790822428$ \\
%$3$ & 7 & 2 & $0.262419912279099$ & $0.2$ & $3.01918455125416$ \\
%$4$ & 7 & 2 & $0.267971525593596$ & $0.2$ & $2.99992262616473$ \\
%$5$ & 7 & 2 & $0.269778216252938$ & $0.2$ & $2.99366741096302$ \\
%$6$ & 6 & 2 & $0.270420675707458$ & $0.2$ & $2.99144462913955$ \\
$0$ & 4 & 3 & $0.20527$ & $0.2$ & $3.2211$ \\
\midrule
$1$ & 5 & 3 & $0.24654$ & $0.2$ & $3.0746$ \\
$2$ & 5 & 2 & $0.24712$ & $0.2$ & $3.0726$ \\
$3$ & 5 & 2 & $0.26242$ & $0.2$ & $3.0192$ \\
$4$ & 5 & 2 & $0.26797$ & $0.2$ & $2.9999$ \\
$5$ & 5 & 2 & $0.26978$ & $0.2$ & $2.9937$ \\
$6$ & 4 & 2 & $0.27042$ & $0.2$ & $2.9914$ \\
\midrule
$\infty$ & n/a & n/a & $0.27181$ & $0.2$ & $2.9866$\\
\bottomrule
\end{tabular}
\caption{Comparison of the solutions of problem \eqref{eq:ex3} with different moment constraints. $m=0$ is conventional robust optimization, $m=\infty$ corresponds to conventional stochastic programming. Intermediate values of $m$ yield solutions at different levels of risk-aversion. The  solutions were obtained using Algorithm \ref{alg:CONCUSSION}, with constant centering $s^{(k)}=10^{-3}$, and stopping condition $\sigma<10^{-8}$, except for $m=\infty$ (see text).}\label{tbl:ex3}
\end{table}

\subsubsection{A portfolio optimization example}

We illustrate the use of Algorithms \ref{alg:CONCUSSION} and \ref{alg:maxP} for the solution of \eqref{eq:DRO} using a portfolio optimization example motivated by \citep{DelageYe2010}. In our experiments we randomly chose three assets from the 30 Dow Jones assets, and tracked for a year the performance of a dynamically allocated portfolio that was rebalanced daily. Each day the 30-day history of the assets were used to estimate the moments of the return distribution, and reallocate the portfolio according to following the optimal moment-robust distribution.

We split the results into two parts: we carried out the simulation using both 2008 and 2009 data to study the properties of the optimal portfolios under very different market conditions (hectic and generally downward in 2008, versus strongly increasing in 2009). In both cases we looked at portfolios optimized using different moment constraints (or, using the notation of Example \ref{ex:3}, we used different sets $\mathfrak{P}_m$). We tracked a portfolio optimized using only first and second moment constraints, and one where the third and fourth marginal moments were also constrained. Sample plots are shown in Fig.~\ref{fig:returns}, where the selected assets were AXP, HPQ, and IBM.

The results show the anticipated trends: the more conservative portfolio (optimized for the worst case among all return distributions compatible with the observed first and second moments) invests generally less, and avoids big losses better than the second portfolio (which is optimized for the worse case among a smaller set of distributions), at the price of missing out on a larger possible return.

\begin{figure}
\centering
\subfigure[year 2008]{
  \includegraphics[width=200pt]{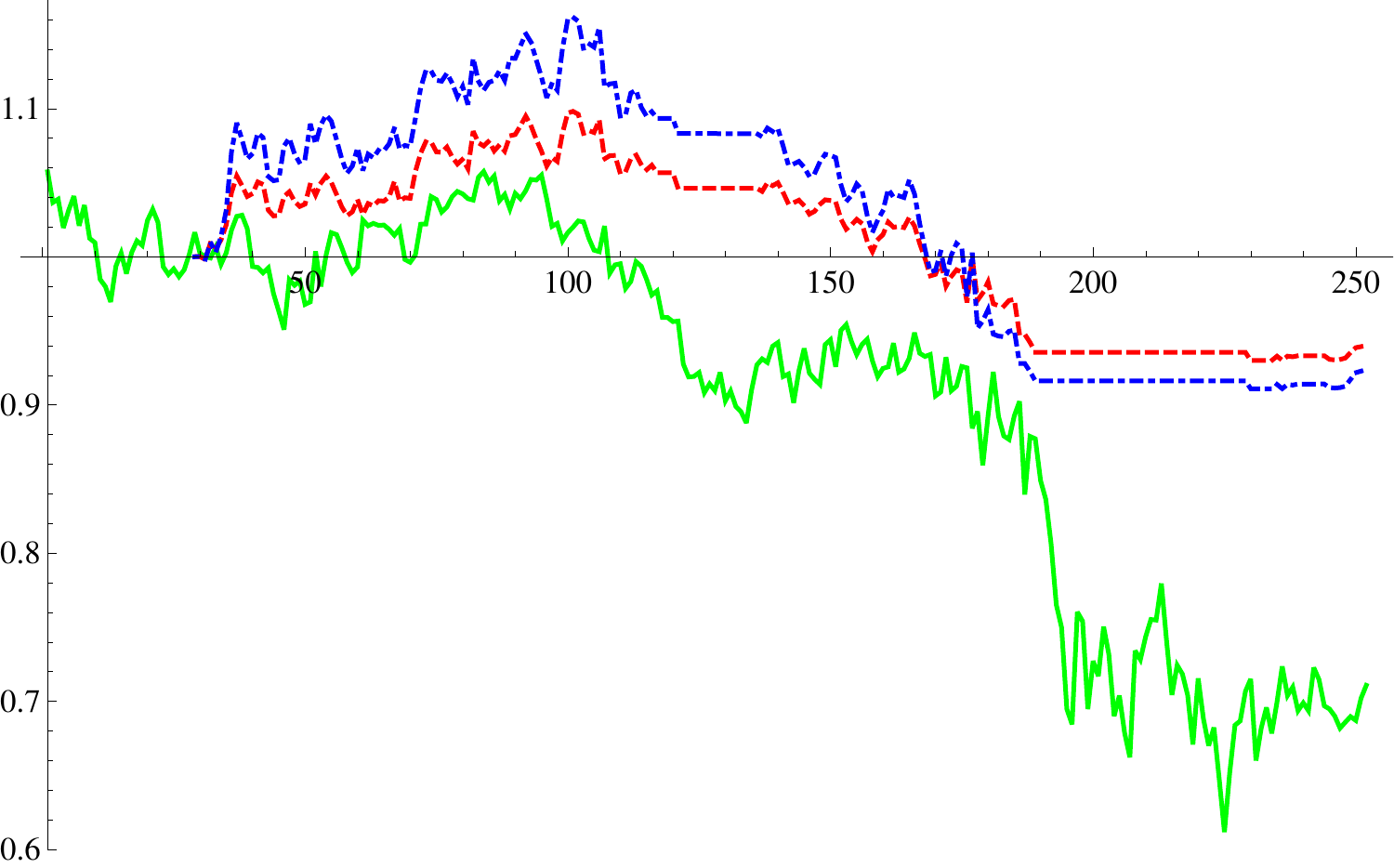}
}
\subfigure[year 2009]{
  \includegraphics[width=200pt]{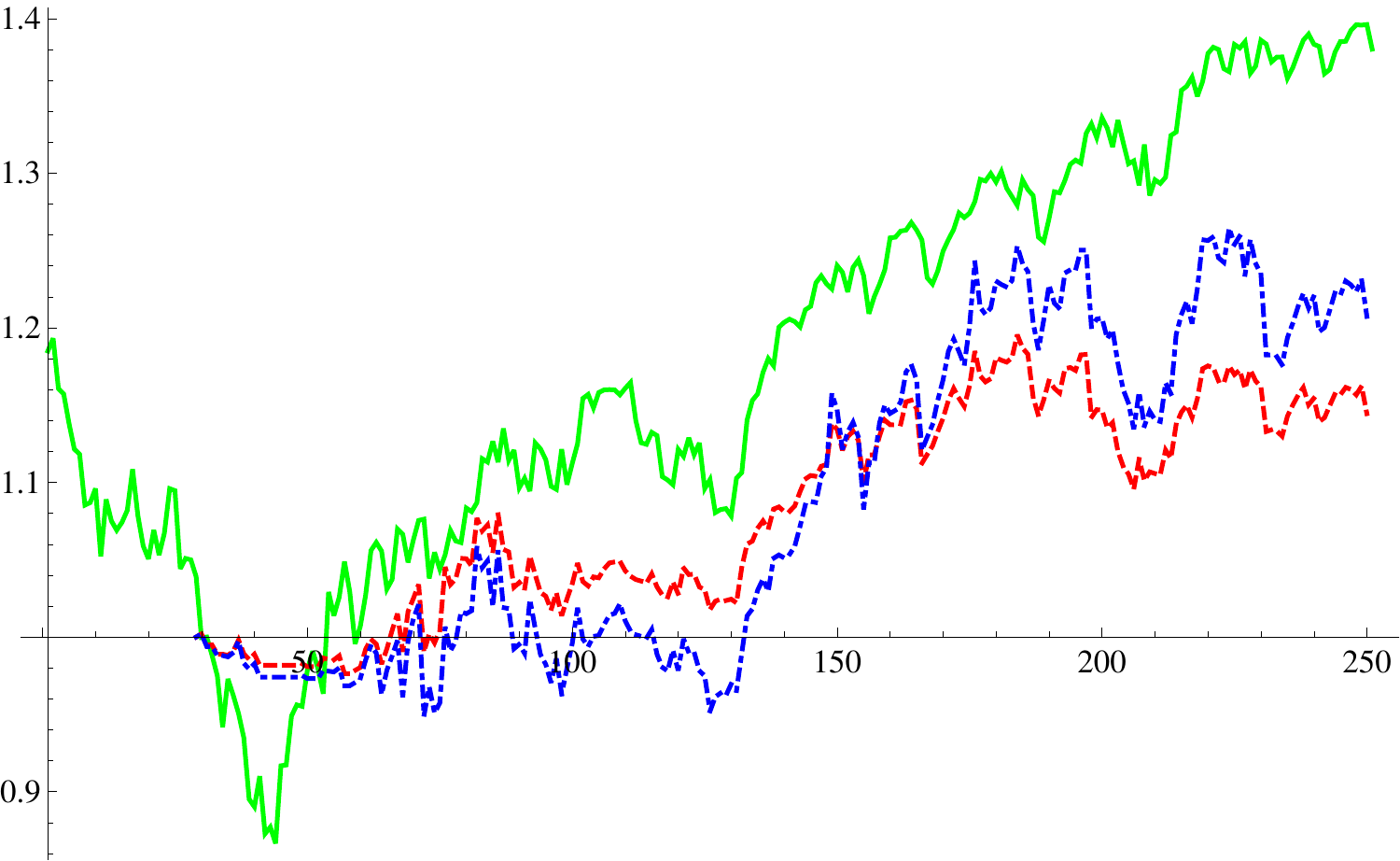}
}
\caption{The performance of two moment-robust portfolios rebalanced daily, compared to market performance. The market (solid, green) is the Dow Jones index scaled to have value $1$ at the start of the experiment (day 31). The red dashed line shows the value of a portfolio optimized using the first and second moment information of the last 30 days' return. (Hence the curve starts at day 31.) The blue dot-dashed line shows the value of a portfolio optimized using the same moments and also the third and fourth marginal moments of the last 30 days' return. As expected, the first, more conservative portfolio outperforms the second one whenever the market conditions are bad, and only then. Both robust portfolios avoid the sharp drop in 2008 by not investing.}\label{fig:returns}
\end{figure}

The algorithm was implemented in Matlab R2012a (Windows 7 64-bit), using the interior-point solver IPOPT 3.10.2 for the solution of the master problems and the linear programming solver CPLEX 12.5 for the cut generation oracle subproblems, and was run on a desktop computer with an Intel Xeon 3.06GHz CPU. Tbls.~\ref{tbl:dro_stats_2} and \ref{tbl:dro_stats_4} show the summary statistics of the algorithms performance, separately for the instances with up to second moment constraints and for the instances with moment constraints of order up to 4. The stopping criterion for the cutting surface algorithm was $\sigma < 10^{-3}$.

\begin{table}[!h]
\centering
\begin{tabular}{llllll}
\toprule
                          & min & 25\% & median & 75\% & max \\
\midrule
master problem time [sec] & 0.1708 & 0.52 & \phantom{1}0.77 & \phantom{1}1.14 & \phantom{11}2.05\\
master problem iterations & 2 & 2 &  \phantom{1}4 &  \phantom{1}5 & \phantom{1}10\\
subproblem time [sec]     & 0.0140 & 1.47 & 21.28 & 43.77 & 180.77\\
subproblem iterations     & 1 &  27  & 54 & 87.25 & 186 \\
\midrule
total wall-clock time [sec] & 9.4851 & 19.854 & 75.089 & 109.81 & 312\\
\bottomrule
\end{tabular}
\caption{Summary statistics of the moment robust optimization algorithm on the portfolio optimization example with moment constraints up to order 2. Each problem instance corresponds to one day in year 2008 or 2009; the table shows iteration count and timing results per instance.}\label{tbl:dro_stats_2}
\end{table}

\begin{table}[!h]
\centering
\begin{tabular}{llllll}
\toprule
                          & min & 25\% & median & 75\% & max \\
\midrule
master problem time [sec] & 0.2049 & 0.412 & \phantom{1}0.602 & \phantom{1}0.775 & \phantom{1}1.04\\
master problem iterations & 2 & 2 &  \phantom{1}2 &  \phantom{1}2 & \phantom{1}5\\
subproblem time [sec]     & 0.0182 & 9.551 & 15.1  & 28.1  & 88.2\\
subproblem iterations     & 1  &   3  &  43  &  87 &  986 \\
\midrule
total wall-clock time [sec] & 9.6945 & 11.192 & 17.666 & 29.738 & 136.87\\
\bottomrule
\end{tabular}
\caption{Summary statistics of the moment robust optimization algorithm on the portfolio optimization example with moment constraints up to order 4.}\label{tbl:dro_stats_4}
\end{table}

As expected, the bottleneck of the algorithm is the randomized cut generation oracle: Algorithm \ref{alg:maxP} takes considerably longer time to find a distribution whose corresponding constraint is violated than it takes to solve the master problems, which are very small convex optimization problems. Nevertheless, the cutting surface algorithm achieved very fast convergence (requiring less then 5 iterations for most instances), and therefore most problems were solvable within one minute.

\section{Conclusion}\label{sec:conclusion}

The convergence of the central cutting surface algorithm was proven under very mild assumptions, which are essential to keep the problem at hand convex, with a non-empty interior. The possibility of using non-differentiable functions in the constraints whose subgradients may not be available, as well as using an infinite dimensional constraint set, may extend the applicability of semi-infinite programming to new territories.
We found that the number of surface cuts can be considerably lower than the number of linear cuts in cutting plane algorithms, which compensates for having to solve a convex optimization problem in each iteration instead of a linear programming problem.
We also found the choice of the centering parameter to be less important for the cutting surface algorithm; both cutting surface and cutting plane algorithms benefit from the more general analysis of our paper, which allows different choices for this parameter from the gradient norm proposed by Kortanek and No.

Our main motivation was distributionally robust optimization, but we hope that other applications involving constraints on probability distributions, and other problems involving a high-dimensional index set $T$, will be forthcoming.

Distributionally robust optimization with multivariate distributions is a relatively recent area, where not even the correct algorithmic framework to handle the arising problems can yet be agreed upon. Methods proposed in the most recent literature include interior point methods for semidefinite programming and the ellipsoid method, but these are not applicable in the presence of moment constraints of order higher than two. Our algorithm is completely novel in the sense that it is the first semi-infinite programming approach to distributionally robust optimization, and it is also the most generally applicable algorithm proposed to date.

Although it can hardly be expected that the semi-infinite programming based approach will be as efficient as the polynomial time methods proposed for the special cases, further research into moment matching scenario generation and distribution optimization algorithms may improve on the efficiency of our method. Simple heuristics might also be beneficial. For example, if several cuts (corresponding to probability distributions $P_1,\dots,P_k$) have already been found and added to the master problem, then before searching for the next cut among distributions supported on the whole domain $\Xi$, we can first search among distributions supported on the union of the support of the distributions $P_1,\dots,P_k$. This is a considerably cheaper step, which requires only the solution of a (finite) linear program, whose solution can be further accelerated by warmstarting.

Since without third and fourth moment information the overall shape of a distribution cannot be determined even approximately, we expect that future successful algorithms in distributionally robust optimization will also have the ability of including higher order moment information in the definition of the uncertainty sets.

\section*{Acknowledgements} The research was partially supported by the grant NSF CMMI-1100868. This material is based upon work supported by the U.S.~ Department of Energy Office of Science, Office of Advanced Scientific Computing Research, Applied Mathematics program under Award Number DOE-SP0011568.

We are also grateful to the referees for their input on both technical details and the presentation of the material.

\appendix

\section*{Appendix: Proof of Theorem \ref{thm:DRO-to-SP}}

\begin{proof}
Let $z_i$ denote the optimal objective function value of \eqref{eq:DROi} for every $i$, and let $z_{SP}$ denote the optimal objective function value of \eqref{eq:SP}; we want to show that $\lim_{i\to\infty}z_i = z_{SP}$.

The sequence $(z_i)_{i=0,1,\dots}$ is convergent because it is monotone decreasing (since $\mathfrak{P}_0 \supseteq \mathfrak{P}_1 \supseteq \cdots \supseteq \cap_{i=0}^m\mathfrak{P}_i$) and it is bounded from below by $z_{SP}$:
\begin{equation}\label{eq:zigeqzsp}
\begin{split}
 z_i &= \min_{x\in X} \max_{Q\in\mathfrak{P}_i} \int_{\xi\in\Xi} h(x,\xi) Q(d\xi) \geq \max_{Q\in\mathfrak{P}_i} \min_{x\in X} \int_{\xi\in\Xi} h(x,\xi) Q(d\xi)\\
&\geq \min_{x\in X} \int_{\xi\in\Xi} h(x,\xi) P(d\xi) = z_{SP}.
\end{split}
\end{equation}

Consider now the stochastic programming problem \eqref{eq:SP}. Denote by $\bar{x}$ one of its optimal solutions, and let
\[\bar{z}_i \defeq \max_{Q\in\mathfrak{P}_i} \int_{\xi\in\Xi} h(\bar{x},\xi)Q(d\xi).\]
Obviously, $z_i \leq \bar{z}_i$ for every $i$. In view of \eqref{eq:zigeqzsp}, it suffices to show that $\bar{z}_i \to z_{SP}$.

For every $i$, choose an arbitrary $\bar{Q}_i \in \argmax_{Q\in\mathfrak{P}_i} \int_{\xi\in\Xi} h(\bar{x},\xi)Q(d\xi)$. Since the moments of $\bar{Q}_i$ and $P$ agree up to order $i$, we have that
\begin{equation}\label{eq:bize}
\int_{\xi\in\Xi}p(\xi) \bar{Q}_i(d\xi) = \int_{\xi\in\Xi}p(\xi) P(d\xi)
\end{equation}
for every polynomial $p$ of total degree at most $i$.

% By definition, the sequence of measures $\bar{Q}_i$ converges to $P$ in moments, that is,
%\[\lim_{i\to\infty}m_k(\bar{Q}_i) = m_k(P) \quad \text{for every multi-index $k\in\mathbb{N}^n$},\]
%which implies that for every polynomial

By assumption, the function $h(\bar{x},\cdot)$ is continuous on the closed and bounded set $\Xi$. Let $p_j$ denote its best uniform polynomial approximation of total degree $j$; by the Weierstrass approximation theorem we have that for every $\epsilon>0$ there exists a degree $j(\epsilon)$ such that $\max_{\xi\in\Xi} |h(\bar{x},\xi)-p_{j(\epsilon)}(\xi)| < \epsilon$, and therefore,
\begin{equation}\label{eq:ize}
\int_{\xi\in\Xi} |h(\bar{x},\xi)-p_{j(\epsilon)}(\xi)|\bar{Q}_i(d\xi) < \epsilon \text{ and } \int_{\xi\in\Xi} |h(\bar{x},\xi)-p_{j(\epsilon)}(\xi)|P(d\xi) < \epsilon.
\end{equation}
With this $j(\epsilon)$, every $i\geq j(\epsilon)$ satisfies the inequalities
\begin{equation*}
\begin{split}
&|\bar{z}_i-z_{SP}| = \left|\int_{\xi\in\Xi} h(\bar{x},\xi)\bar{Q}_i(d\xi) -\int_{\xi\in\Xi} h(\bar{x},\xi)P(d\xi)\right|\leq\\
&\leq \left|\int_{\xi\in\Xi}h(\bar{x},\xi)\bar{Q}_i(d\xi) - \int_{\xi\in\Xi}p_{j(\epsilon)}(\xi)\bar{Q}_i(d\xi)\right| +
\left|\int_{\xi\in\Xi}p_{j(\epsilon)}(\xi)\bar{Q}_i(d\xi) - \int_{\xi\in\Xi}h(\bar{x},\xi)P(d\xi)\right|=\\
&=\left|\int_{\xi\in\Xi}\left(h(\bar{x},\xi) - p_{j(\epsilon)}(\xi)\right)\bar{Q}_i(d\xi)\right| +
\left|\int_{\xi\in\Xi}p_{j(\epsilon)}(\xi)P(d\xi) - \int_{\xi\in\Xi}h(\bar{x},\xi)P(d\xi)\right|\leq\\
&\leq \int_{\xi\in\Xi}\left|h(\bar{x},\xi) - p_{j(\epsilon)}(\xi)\right|\bar{Q}_i(d\xi) +
\int_{\xi\in\Xi}\left|h(\bar{x},\xi) - p_{j(\epsilon)}(\xi)\right|P(d\xi) < 2\epsilon,
\end{split}
\end{equation*}
using the triangle inequality, \eqref{eq:bize}, \eqref{eq:ize}, and the triangle inequality again.
From the inequality between the left- and the right-hand side it immediately follows that $\lim_{i\to\infty}\bar{z}_i = z_{SP}$, as claimed.
\end{proof}

\bibliographystyle{abbrvnat}
\bibliography{SICP_cutting}

\begin{thebibliography}{24}
\providecommand{\natexlab}[1]{#1}
\providecommand{\url}[1]{\texttt{#1}}
\expandafter\ifx\csname urlstyle\endcsname\relax
  \providecommand{\doi}[1]{doi: #1}\else
  \providecommand{\doi}{doi: \begingroup \urlstyle{rm}\Url}\fi

\bibitem[Bertsimas et~al.(2010)Bertsimas, Doan, Natarajan, and
  Teo]{Bertsimas2010models}
D.~Bertsimas, X.~V. Doan, K.~Natarajan, and C.-P. Teo.
\newblock Models for minimax stochastic linear optimization problems with risk
  aversion.
\newblock \emph{Mathematics of Operations Research}, 35\penalty0 (3):\penalty0
  580--602, 2010.

\bibitem[Betr{\`o}(2004)]{Betro2004}
B.~Betr{\`o}.
\newblock An accelerated central cutting plane algorithm for linear
  semi-infinite programming.
\newblock \emph{Mathematical Programming}, 101:\penalty0 479--495, 2004.
\newblock \doi{10.1007/s10107-003-0492-5}.

\bibitem[de~Klerk(2008)]{dK2008}
E.~de~Klerk.
\newblock The complexity of optimizing over a simplex, hypercube or sphere: a
  short survey.
\newblock \emph{Central European Journal of Operations Research}, 16\penalty0
  (2):\penalty0 111--125, 2008.
\newblock ISSN 1435-246X.
\newblock \doi{10.1007/s10100-007-0052-9}.
\newblock URL \url{http://dx.doi.org/10.1007/s10100-007-0052-9}.

\bibitem[de~Klerk et~al.(2006)de~Klerk, Laurent, and Parrilo]{dKLP2006}
E.~de~Klerk, M.~Laurent, and P.~A. Parrilo.
\newblock A {PTAS} for the minimization of polynomials of fixed degree over the
  simplex.
\newblock \emph{Theoretical Computer Science}, 361\penalty0 (2--3):\penalty0
  210--225, 2006.
\newblock ISSN 0304-3975.
\newblock \doi{http://dx.doi.org/10.1016/j.tcs.2006.05.011}.

\bibitem[de~Loera et~al.(2008)de~Loera, Hemmecke, K{\"o}ppe, and
  Weismantel]{dLHKW2008}
J.~A. de~Loera, R.~Hemmecke, M.~K{\"o}ppe, and R.~Weismantel.
\newblock {FPTAS} for optimizing polynomials over the mixed-integer points of
  polytopes in fixed dimension.
\newblock \emph{Mathematical Programming}, 115\penalty0 (2):\penalty0 273--290,
  2008.
\newblock ISSN 0025-5610.
\newblock \doi{10.1007/s10107-007-0175-8}.
\newblock URL \url{http://dx.doi.org/10.1007/s10107-007-0175-8}.

\bibitem[Delage and Ye(2010)]{DelageYe2010}
E.~Delage and Y.~Ye.
\newblock Distributionally robust optimization under moment uncertainty with
  application to data-driven problems.
\newblock \emph{Operations Research}, 58:\penalty0 595--612, 2010.
\newblock \doi{10.1287/opre.1090.0741}.

\bibitem[Gribik(1979)]{Gribik1979}
P.~R. Gribik.
\newblock A central cutting plane algorithm for semi-infinite programming
  problems.
\newblock In R.~Hettich, editor, \emph{Semi-infinite programming}, number~15 in
  Lecture Notes in Control and Information Systems. Springer Verlag, New York,
  {NY}, 1979.

\bibitem[Henrion and Lasserre(2003)]{gloptipoly}
D.~Henrion and J.-B. Lasserre.
\newblock {GloptiPoly}: Global optimization over polynomials with {M}atlab and
  {SeDuMi}.
\newblock \emph{{ACM} Transactions on Mathematical Software}, 29\penalty0
  (2):\penalty0 165--194, June 2003.
\newblock ISSN 0098-3500.
\newblock \doi{10.1145/779359.779363}.

\bibitem[H{\o}yland et~al.(2003)H{\o}yland, Kaut, and Wallace]{HKW-03}
K.~H{\o}yland, M.~Kaut, and S.~W. Wallace.
\newblock A heuristic for moment-matching scenario generation.
\newblock \emph{Computational Optimization and Applications}, 24\penalty0
  (2):\penalty0 169--185, Feb. 2003.
\newblock \doi{10.1023/A:1021853807313}.

\bibitem[Huang and Mehrotra(2013)]{HuangMehrotra2013}
K.-L. Huang and S.~Mehrotra.
\newblock An empirical evaluation of walk-and-round heuristics for mixed
  integer linear programs.
\newblock \emph{Computational Optimization and Applications}, 2013.
\newblock \doi{10.1007/s10589-013-9540-0}.

\bibitem[Kannan and Narayanan(2012)]{KannanNarayanan2012}
R.~Kannan and H.~Narayanan.
\newblock Random walks on polytopes and an affine interior point method for
  linear programming.
\newblock \emph{Mathematics of Operations Research}, 37\penalty0 (1):\penalty0
  1--20, Feb. 2012.
\newblock \doi{10.1016/j.ejor.2006.08.045}.

\bibitem[Kleiber and Stoyanov(2013)]{KleiberStoyanov2013}
C.~Kleiber and J.~Stoyanov.
\newblock Multivariate distributions and the moment problem.
\newblock \emph{Journal of Multivariate Analysis}, 113\penalty0 (1):\penalty0
  7--18, 2013.
\newblock ISSN 0047-259X.
\newblock \doi{10.1016/j.jmva.2011.06.001}.
\newblock URL \url{http://dx.doi.org/10.1016/j.jmva.2011.06.001}.

\bibitem[Kortanek and No(1993)]{KortanekNo1993}
K.~O. Kortanek and H.~No.
\newblock A central cutting plane algorithm for convex semi-infinite
  programming problems.
\newblock \emph{SIAM Journal on Optimization}, 3\penalty0 (4):\penalty0
  901--918, Nov. 1993.

\bibitem[Li(2011)]{lithesis2011}
Z.~Li.
\newblock \emph{Polynomial Optimization Problems -- Approximation Algorithms
  and Applications}.
\newblock PhD thesis, The Chinese University of Hong Kong, 2011.

\bibitem[L{\'o}pez and Still(2007)]{LopezStill2007}
M.~L{\'o}pez and G.~Still.
\newblock Semi-infinite programming.
\newblock \emph{European Journal of Operational Research}, 180:\penalty0
  491--518, 2007.
\newblock \doi{10.1016/j.ejor.2006.08.045}.

\bibitem[Lov{\'a}sz and Vempala(2006)]{LovaszVempala2006}
L.~Lov{\'a}sz and S.~Vempala.
\newblock Hit-and-run from a corner.
\newblock \emph{SIAM Journal on Computing}, 35\penalty0 (4):\penalty0
  985--1005, 2006.
\newblock \doi{10.1137/S009753970544727X}.

\bibitem[Mehrotra and Papp(2013)]{MehrotraPapp2013}
S.~Mehrotra and D.~Papp.
\newblock Generating moment matching scenarios using optimization techniques.
\newblock \emph{SIAM Journal on Optimizaton}, 23\penalty0 (2):\penalty0
  963--999, 2013.
\newblock URL \url{http://dx.doi.org/10.1137/110858082}.

\bibitem[Mehrotra and Zhang(2013)]{MehrotraZhang2013}
S.~Mehrotra and H.~Zhang.
\newblock Models and algorithms for distributionally robust least squares
  problems.
\newblock \emph{Accepted in Mathematical Programming}, 2013.
\newblock URL \url{http://link.springer.com/article/10.1007/s10107-013-0681-9}.
\newblock Technical report URL:
  \url{http://www.optimization-online.org/DB_FILE/2011/02/2925.pdf}.

\bibitem[Papp and Alizadeh(2011)]{PappAlizadeh2011}
D.~Papp and F.~Alizadeh.
\newblock Semidefinite characterization of sum-of-squares cones in algebras.
\newblock \emph{Accepted in SIAM Journal on Optimization}, 2011.

\bibitem[Parrilo(2003)]{Parrilo2003}
P.~A. Parrilo.
\newblock Semidefinite programming relaxations for semialgebraic problems.
\newblock \emph{Mathematical Programming}, 96\penalty0 (2):\penalty0 293--320,
  2003.
\newblock ISSN 0025-5610.
\newblock \doi{10.1007/s10107-003-0387-5}.

\bibitem[Scarf(1957)]{Scarf1957}
H.~E. Scarf.
\newblock A min-max solution of an inventory problem.
\newblock Technical Report P-910, The {RAND} Corporation, 1957.

\bibitem[Tichatschke and Nebeling(1988)]{TichatschkeNebeling1988}
R.~Tichatschke and V.~Nebeling.
\newblock A cutting-plane method for quadratic semi infinite programming
  problems.
\newblock \emph{Optimization}, 19\penalty0 (6):\penalty0 803--817, 1988.
\newblock \doi{10.1080/02331938808843393}.

\bibitem[Vempala(2005)]{Vempala2005}
S.~Vempala.
\newblock Geometric random walks: a survey.
\newblock \emph{Combinatorial and Computational Geometry}, 52:\penalty0
  573--612, 2005.

\bibitem[Waki et~al.(2006)Waki, Kim, Kojima, and Muramatsu]{sparsepop}
H.~Waki, S.~Kim, M.~Kojima, and M.~Muramatsu.
\newblock Sums of squares and semidefinite programming relaxation for
  polynomial optimization problems with structured sparsity.
\newblock \emph{{SIAM} Journal on Optimization}, 17:\penalty0 218--242, 2006.
\newblock ISSN 1052-6234.
\newblock \doi{10.1137/050623802}.

\end{thebibliography}
\end{document}